\documentclass[reqno, 11pt]{amsart}
\usepackage{amsmath}	
\usepackage{amssymb}	
\usepackage{amsthm}
\usepackage{amsfonts}
\usepackage{latexsym}
\usepackage{mathrsfs}
\usepackage[all,ps,cmtip]{xy}
\usepackage{tabularx}
\usepackage{pict2e}
\usepackage{graphicx}
\usepackage{tikz-cd}
\usepackage{tikz}
\usetikzlibrary{matrix,arrows}
\usepackage[pagebackref]{hyperref}

\usepackage[text={155mm, 235mm},centering]{geometry}
\input diagxy
\input xy

\theoremstyle{plain}
\newtheorem{thm}{Theorem}[section]
\newtheorem{lem}[thm]{Lemma}
\newtheorem{prop}[thm]{Proposition}
\newtheorem{cor}[thm]{Corollary}
\newtheorem{claim}[thm]{Claim}

\newtheorem{prob}[thm]{Problem}
\theoremstyle{definition}
\newtheorem{defn}[thm]{Definition}
\newtheorem{rem}[thm]{Remark}
\newtheorem{example}[thm]{Example}

\DeclareMathOperator{\V}{\mathcal{V}}

\DeclareMathOperator{\Z}{\mathbb{Z}}
\DeclareMathOperator{\CO}{\mathcal{O}}
\DeclareMathOperator{\CN}{\mathbb{C}}
\DeclareMathOperator{\K}{\mathcal{K}}
\DeclareMathOperator{\D}{\mathrm{D}}
\DeclareMathOperator{\E}{\mathcal{E}}
\DeclareMathOperator{\F}{\mathcal{F}}

\DeclareMathOperator{\Hom}{\mathrm{Hom}}
\DeclareMathOperator{\Id}{\mathrm{id}}
\DeclareMathOperator{\Coh}{\mathrm{Coh}}

\DeclareMathOperator{\DR}{\mathrm{DR}}
\DeclareMathOperator{\Hdg}{\mathrm{Hdg}}
\DeclareMathOperator{\HC}{\mathcal{H}\mathcal{H}}

\allowdisplaybreaks

\numberwithin{equation}{section}
\begin{document}

\title{Hodge cohomology on blow-ups along subvarieties}

\author[S. Rao]{Sheng Rao}
\address{School of Mathematics and Statistics, Wuhan  University,
Wuhan 430072, P.R. China;
Universit\'{e} de Grenoble-Alpes, Institut Fourier (Math\'{e}matiques)
UMR 5582 du C.N.R.S., 100 rue des Maths, 38610 Gi\`{e}res, France}
\email{likeanyone@whu.edu.cn, sheng.rao@univ-grenoble-alpes.fr}

\author[S. Yang]{Song Yang}
\address{Center for Applied Mathematics, Tianjin University, Tianjin 300072, P.R. China}%
\email{syangmath@tju.edu.cn}%

\author[X. Yang]{Xiangdong Yang}
\address{Department of mathematics, Lanzhou University, Lanzhou 730000, P.R. China}
\email{yangxd@lzu.edu.cn}

\author[X. Yu]{Xun Yu}
\address{Center for Applied Mathematics, Tianjin University, Tianjin 300072, P.R. China}%
\email{xunyu@tju.edu.cn}%

\date{\today}

\begin{abstract}
We establish a blow-up formula for Hodge cohomology of locally free sheaves
on smooth proper varieties over an algebraically closed field of positive
characteristic.
For this, we introduce a notion of relative Hodge sheaves and study their  behavior under blow-ups along smooth centers.
In particular, as an application, we study the blow-up invariance of the $E_2$-degeneracy of the Hochschild--Kostant--Rosenberg spectral sequence for smooth proper varieties.
\end{abstract}


\maketitle

\setcounter{tocdepth}{1}
\tableofcontents

%=================================================================

\section{Introduction}

\subsection{Motivation and results}
Let $X$ be a smooth proper variety over an algebraically closed field $k$ of arbitrary characteristic.
Consider the {\it Hodge--de Rham spectral sequence}
\begin{equation}\label{Hodge-to-de Rham}
E_{1}^{p,q}=H^{q}(X, \Omega_{X}^{p})\Longrightarrow H_{\DR}^{p+q}(X/k),
\end{equation}
where $H_{\DR}^{\bullet}(X/k)$ is the algebraic de Rham cohomology
and $\Omega_{X}^{p}$ is the sheaf of regular differential $p$-forms on $X$.
If $k=\CN$, we set $X^{an}$ as the associated compact complex manifold of $X$.
If $X$ is {\it projective},
then $X^{an}$ is {K\"{a}hler};
by Hodge theory and Serre's GAGA \cite{Ser56},
the Hodge symmetry and the $E_{1}$-degeneracy of Hodge--de Rham spectral sequence \eqref{Hodge-to-de Rham} holds here.
If $X$ is {\it non-projective}, then $X^{an}$ is {non-K\"{a}hler};
based on the Chow's Lemma and Hironaka's resolution of singularities \cite{Hir64},
Deligne \cite[(5.3)]{Del68} showed the Hodge symmetry and the $E_{1}$-degeneracy of \eqref{Hodge-to-de Rham} hold then; intrinsically, such $X^{an}$ is a Moishezon manifold and thus satisfies the $\partial\bar{\partial}$-Lemma which in turn implies the Hodge symmetry and the $E_{1}$-degeneracy (cf. \cite{DGMS75,Voi02}).
More generally, if $k$ is of characteristic $\mathrm{char}(k)=0$, using Lefschetz principle,
Deligne \cite[(5.5)]{Del68} showed that the Hodge symmetry and the $E_{1}$-degeneracy also hold; see also Deligne--Illusie \cite[2.7]{DL87}.
Furthermore, Deligne--Illusie \cite{DL87} showed that the $E_{1}$-degeneracy still holds if $\mathrm{char}(k)\geq \dim\, X$ and $X$ lifts to the ring
$W_{2}(k)$ of Witt vectors of length $2$.
However, in general,
the Hodge symmetry and the $E_{1}$-degeneracy fail in positive characteristic and for compact complex manifolds.
In \cite[Theorem 1.6]{RYY19a}, we derived the blow-up invariance of the $E_{1}$-degeneracy of \eqref{Hodge-to-de Rham} for compact complex manifolds by obtaining a blow-up formula of Dolbeault cohomology.
Let $Z$ be a smooth closed subvariety of $X$ and $\tilde{X}$ the blow-up of $X$ along the center $Z$.
By the {\it blow-up invariance} of a property, we mean that such a property holds for $X$ and $Z$ if and only if so does for the blowing up variety $\tilde{X}$.
In positive characteristic, although the liftability to $W_{2}(k)$ possibly fails under blow-ups of higher dimensional varieties by Liedtke--Satriano \cite{LS14},
Achinger--Zdanowicz \cite[Corollary 2.9.(1)]{AZ17} obtained the blow-up invariance of the $E_{1}$-degeneracy of \eqref{Hodge-to-de Rham} by using the blow-up formulae of algebraic de Rham and total Hodge cohomologies via Voevodsky's blow-up formula \cite[(3.5.3)]{Voe00} of motives.

In a more general setting, consider a locally free sheaf $\mathcal{E}$ on $X$ which is endowed with an integrable connection.
We denote by $H_{\DR}^{l}(X/k; \mathcal{E})$ the {\it $l$-th algebraic de Rham cohomology of $X$ with coefficients in $\E$} (see \cite{Gro66}).
Then one has the {\it twisted Hodge--de Rham spectral sequence}:
\begin{equation}\label{twisted-Hodge-to-de Rham}
E_{1}^{p,q}=H^{q}(X, \Omega_{X}^{p}\otimes \E)\Longrightarrow H_{\DR}^{p+q}(X/ k; \E).
\end{equation}
In general, the twisted Hodge--de Rham spectral sequence does not degenerate at the $E_{1}$-term.
So a natural problem is:

\begin{prob}\label{questin-0}
 Does the blow-up invariance of the $E_{1}$-degeneracy for the twisted Hodge--de Rham spectral sequence hold for smooth proper varieties over $k$ of positive characteristic?
\end{prob}

Motivated by Problem \ref{questin-0}, we prove a blow-up formula for Hodge cohomology of locally free sheaves.
For this purpose
we will introduce a notion of relative Hodge sheaves associated to the pair of a smooth proper variety and its smooth closed subvariety.
Suppose that $X$ is a smooth proper variety together with a smooth closed subvariety
$j: Y\hookrightarrow X$.
The kernel sheaf $\ker\big(\Omega_{X}^{p} \stackrel{j^{\#}}\rightarrow  j_{\ast}\Omega_{Y}^{p}\big)$ is called the {\it $p$-th relative Hodge sheaf $\K_{Y}^{p}$} of $X$ with respect to $Y$ (Definition \ref{rhs}).
Via studying explicitly relations of relative Hodge sheaves under blow-ups,
we derive the main result.

\begin{thm}\label{mainthm}
Let $X$ be an $n$-dimensional smooth proper variety over an algebraically closed field $k$ of arbitrary characteristic,
$\iota:Z\hookrightarrow X$ a smooth closed subvariety of codimension $c\geq 2$
and $\pi: \tilde{X}\longrightarrow X$ the blow-up of $X$ along $Z$ with the exceptional divisor $E$.
For any locally free sheaf $\V$ on $X$, the following statements are true:
\begin{itemize}
  \item [(i)]For the cohomology of relative Hodge sheaves $\K_{Z}^{p}$ and $\K_{E}^{p}$,
the pullback of regular differential forms induces a natural isomorphism:
\begin{equation}\label{icrhs}
\pi^{\#}:
H^{q}(X, \K_{Z}^{p}\otimes \V)
\stackrel{\simeq}\longrightarrow
H^{q}(\tilde{X}, \K_{E}^{p}\otimes \pi^{\ast}\V),\;\; \textrm{for any }\, 0\leq p,q\leq n.
\end{equation}
\item [(ii)]There exists an isomorphism
\begin{equation}\label{bundle-blowup-formula}
H^{q}(\tilde{X}, \Omega_{\tilde{X}}^{p}\otimes \pi^{\ast}\V)
\cong
H^{q}(X,\Omega_{X}^{p}\otimes \V)
\oplus
\bigoplus_{i=1}^{c-1} H^{q-i}(Z, \Omega_{Z}^{p-i} \otimes \iota^{\ast} \V)
\end{equation}
for any $0\leq p,q\leq n$.
\end{itemize}
\end{thm}

In \cite[Theorem 1.2]{RYY19b}, using a notion of relative Dolbeault sheaves and the Dolbeault resolutions,
we proved the blow-up formula \eqref{bundle-blowup-formula} for compact complex manifolds.
Furthermore, if $X$ is a smooth proper variety over $\mathbb{C}$,
then the isomorphism \eqref{bundle-blowup-formula} can be obtained by Serre's GAGA and \cite[Theorem 1.2]{RYY19b};
hence, the characteristic zero case can be handled by the Lefschetz principle.
On a smooth variety $X$ over an algebraically closed field $k$ of arbitrary characteristic,
there is no Dolbeault resolution of $\Omega_{X}^{p}$ analogous to that on complex manifolds
since the Zariski topology is coarser than the complex topology.
So the proof in \cite{RYY19b} does not hold step-by-step for general smooth varieties;
however, we will see that the basic idea in \cite{RYY19b,YY17} still holds.

In particular, applying Theorem \ref{mainthm} to torsion line bundles provides us with a useful perspective to understand an interesting problem by Esnault--Ogus \cite[Question 2.1]{EO09}, see Lemma \ref{blowup-EO-quest}.
Moreover, if the locally free sheaf $\V$ in Theorem \ref{mainthm} is the structure sheaf,
then we have the blow-up formula for $(p,q)$-Hodge cohomology which is implicitly contained in \cite{Gros85,AZ17}.
Namely, there exists an isomorphism
\begin{equation}\label{Hodge-bwf}
H^{q}(\tilde{X}, \Omega_{\tilde{X}}^{p})
\cong
H^{q}(X,\Omega_{X}^{p})
\oplus
\bigoplus_{i=1}^{c-1} H^{q-i}(Z, \Omega_{Z}^{p-i})
\end{equation}
of Hodge cohomology.

Similarly, we have the {\it Hochschild--Kostant--Rosenberg spectral sequence} (for short {\it HKR spectral sequence}):
\begin{equation*}
E_{2}^{p, q}=H^{q}(X, \Omega_{X}^{p})\Longrightarrow \mathrm{HH}_{p-q}(X),
\end{equation*}
where the differential $d_{r}$ has bi-degree $(r-1, r)$ and $\mathrm{HH}_{\bullet}(X)$ is the Hochschild homology of $X$.
In characteristic zero, the HKR spectral sequence is known to degenerate at $E_{2}$ for smooth proper varieties.
The HKR spectral sequence also degenerates at $E_{2}$ over $k$ of characteristic $\mathrm{char}(k)\geq \dim\, X$ by Yekutieli \cite{Yek02} and Antieau--Vezzosi \cite{AV17}.
However, when $\dim\, X> \mathrm{char}(k)>0$,
Antieau--Bhatt--Mathew \cite{ABM19} showed that the HKR spectral sequence does not generally degenerate at $E_{2}$ (cf. Antieau--Bragg \cite{AB19}).
Moreover,
here arises a similarly natural

\begin{prob}\label{questin}
 Does the blow-up invariance of the $E_{2}$-degeneracy of the HKR spectral sequence hold for smooth proper varieties over $k$ of positive characteristic?
\end{prob}

Based on Orlov's blow-up formula \cite{Orl93} for derived categories and the decomposition of Hochschild cohomology \cite{Kuz09} under semiorthogonal decompositions,
as a direct corollary of \eqref{Hodge-bwf}, we confirm Problem \ref{questin} as:

\begin{thm}[= Corollary \ref{HKR-blowup}]\label{thm2}
With the same setting as in Theorem \ref{mainthm},
the $E_{2}$-degeneracy of the HKR spectral sequence holds for $\tilde{X}$ if and only if it holds for $X$ and $Z$.
\end{thm}

Moreover, the blow-up invariance \cite[Corollary 2.9.(1)]{AZ17} of the $E_{1}$-degeneracy of Hodge--de Rham spectral sequence \eqref{Hodge-to-de Rham} is also a direct result of \eqref{Hodge-bwf} and the blow-up formula \cite[Corollary 2.8.(3)]{AZ17} of algebraic de Rham cohomology.
In particular, based on the examples by Antieau--Bhatt--Mathew, Theorem \ref{thm2} enables us to construct new examples of smooth proper varieties satisfying the non-degeneracy of the HKR spectral sequence at $E_{2}$-page; see Remark \ref{construct-example} and analogous Remark \ref{rem5.25}.

\subsection{Strategy of the proof}
We outline the basic idea of the proof for Theorem \ref{mainthm} as follows.
First of all, we introduce a notion of relative Hodge sheaves $\K_{Z}^{p}$ associated to a pair $(X, Z)$.
Secondly, for the blow-up pairs $(\tilde{X}, E)$ and $(X, Z)$
with $E$ the exceptional divisor,
we shall establish a natural commutative ladder of cohomology:
\begin{equation*}
\xymatrix@C=0.3cm{
\cdots\ar[r]^{} & H^{q}(X, \K^{p}_{Z}\otimes \V)  \ar[d]_{\pi^{\#}} \ar[r]^{} &H^{q}(X, \Omega_{X}^{p}\otimes \V) \ar[d]_{\pi^{\#}} \ar[r]^{} & H^{q}(Z,  \Omega_{Z}^{p}\otimes \iota^{\ast}\V) \ar[d]_{\rho^{\#}} \ar[r]^{} & H^{q+1}(X, \K^{p}_{Z}\otimes \V)\ar[d]_{\pi^{\#}}\ar[r]^{} & \cdots\\
 \cdots\ar[r] & H^{q}(\tilde{X},\K^{p}_{E}\otimes \pi^{\ast}\V) \ar[r]^{} &
  H^{q}(\tilde{X}, \Omega_{\tilde{X}}^{p}\otimes \pi^{\ast}\V) \ar[r]^{} &
  H^{q}(E,  \Omega_{E}^{p}\otimes \tilde{\iota}^{\ast}\pi^{\ast}\V)\ar[r]^{} &
  H^{q+1}(\tilde{X},\K^{p}_{E}\otimes \pi^{\ast}\V) \ar[r] &\cdots. }
\end{equation*}
Furthermore, in the above diagram,
the first and fourth column morphisms
\begin{equation*}
\pi^{\#}:
H^{\bullet}(X, \K_{Z}^{p}\otimes\V)
\longrightarrow
H^{\bullet}(\tilde{X}, \pi^{-1}(\K_{Z}^{p}\otimes\V))
\longrightarrow
H^{\bullet}(\tilde{X}, \K_{E}^{p}\otimes \pi^{\ast}\V)
\end{equation*}
are isomorphisms,  proved in Theorem \ref{key-tech-prop} as an important property of relative Hodge sheaves,
and the second one is injective (Lemma \ref{injective-lem}).
Hence, the third one is injective by the Four Lemma.
Finally, Theorem \ref{mainthm} follows from basic homological algebra and the projective bundle formula of Hodge cohomology of locally free sheaves.

\subsection{Related works}
This is a continuity of our previous works on the blow-up formulae of Dolbeault, Bott--Chern and twisted de Rham cohomologies (\cite{RYY19a,RYY19b,YY17,CY19}).
In \cite{RYY19a,YY17},
from birational point of view,
we tried to understand the birational invariance of the
$\partial\bar{\partial}$-Lemma and the $E_{1}$-degeneracy of \eqref{Hodge-to-de Rham} for compact complex manifolds.
To obtain the blow-up invariance of these two properties, we develop the blow-up formulae of Bott--Chern and Dolbeault cohomologies.
As a consequence,
we obtain that they are birational properties of compact complex threefolds and fourfolds, respectively, by applying the weak factorization theorem \cite{AKMW02,Wlo03}.
Prior to the relative Hodge sheaves,
the notion of {\it relative Dolbeault sheaves} for a pair of a complex manifold and its closed complex submanifold has been introduced in \cite{RYY19b,YY17} and plays a dominant role in \cite{RYY19b,YY17,CY19}.

Before that,
based upon the detailed study of Leray spectral sequences,
the blow-up formula of \'{e}tale cohomology $H_{\mathrm{et}}^{i}(-, \Z_{l})$ on smooth schemes has been obtained in \cite[XVIII, Theorem 2.2.2]{DK73}.
It is worth noticing that the explicit calculation of $R^{i}\pi_{\ast}\Z_{l}$ which is similar to Lemma \ref{key-iso} was also used;
the interested readers may refer to the proof of \cite[Theorem 10]{Ste18}
and \cite[Appendix B]{RYY19b} for a comparison.
Afterwards, using a different method, Barbieri-Viale \cite{Bar97} obtained the blow-up formula of twisted cohomology theory in the sense of Bloch--Ogus \cite{BO74},
e.g., \'{e}tale cohomology, Deligne--Beilinson cohomology
and algebraic de Rham cohomology in characteristic zero.

\subsection{Notation and conventions}
For simplicity, we always assume that $k$ is an algebraically closed field of arbitrary characteristic.
The experiment shows that our arguments are likely to hold over any field of arbitrary characteristic.

Throughout this paper,
a {\it variety} is an integral separated scheme of finite type over $k$ and a {\it locally free sheaf} is of finite rank.
Let $X$ be a smooth proper variety and $\V$ a locally free sheaf on $X$.
We fix some notations for later use:
\begin{itemize}
\item[--] $\CO_{X}$ $=$ the structure sheaf of $X$;

\item[--] $\Omega_{X}$ (or $\Omega_{X}^{1}$) $=$  the cotangent sheaf of $X$;

\item[--] $\Omega_{X}^{p}:=\wedge^{p}\Omega_{X}$ $=$  the sheaf of regular differential $p$-forms on $X$;

\item[--] $\omega_{X}:=\Omega_{X}^{\dim\,X}$ $=$  the canonical sheaf of $X$;

\item[--] $\F(\V):=\F \otimes\V$ $=$  the tensor over $\CO_{X}$,
                           where $\F$ is a sheaf of $\CO_{X}$-module;

\item[--] $h^{l}(X, \F):=\dim\, H^{l}(X, \F)$ $=$  the dimension as $k$-vector space;

\item[--] $\mathbb{P}(\E):=\mathrm{Proj}(\mathrm{Sym} \E^{\vee})$ $=$  the projective bundle of a locally free sheaf $\E$.
\end{itemize}

%=====================================================================

\section{Preliminaries}

In this section, we give a rapid review on some basic results on sheaf cohomology theory (such as Iversen \cite[II. 7]{Ive86} and Kashiwara--Schapira \cite[\S 2.6]{KS94}) and the construction of blow-ups of smooth proper varieties.

\subsection{Sheaf cohomology}
Let
$f:(Y, \mathcal{R}_{Y})\longrightarrow (X, \mathcal{R}_{X})$ be a morphism of ringed spaces.
We denote by $\D^{+}(\mathcal{R}_{X})$ (resp. $\D^{-}(\mathcal{R}_{X})$)
the bounded {\it below} (resp. {\it above}) derived category of $\mathcal{R}_{X}$-modules.
Then one has the following standard functors:

\begin{itemize}
\item[--]
  $Rf_{\ast}: \D^{+}(\mathcal{R}_{Y})\rightarrow \D^{+}(\mathcal{R}_{X})$
  $=$ the right derived functor of the direct image functor $f_{\ast}$;
\item[--]
  $Lf^{\ast}: \D^{-}(\mathcal{R}_{X})\rightarrow \D^{-}(\mathcal{R}_{Y})$ $=$ the left derived functor of the inverse image functor $f^{\ast}$;
\item[--]
  $f^{-1}$ $=$ the topological inverse image functor.
\end{itemize}

Let $X$ and $Y$ be two smooth proper varieties and $f:Y\longrightarrow X$ a morphism between them.
Since the abelian category of $f^{-1}\CO_{X}$-modules has enough injective objects, there exists a right derived functor of the direct image $f_{\ast}$ denoted by
$$
Rf_{\ast}: \D^{+}(f^{-1}\CO_{X})\longrightarrow \D^{+}(\CO_{X}).
$$
Note that the topological inverse image functor $f^{-1}$ is exact.
It naturally extends to a functor on derived categories
$$
f^{-1}: \D^{+}(\CO_{X})\longrightarrow \D^{+}(f^{-1}\CO_{X}),
$$
which is the \emph{left adjoint} of the derived functor $Rf_{\ast}$.
Therefore, for any objects $\E^{\bullet} \in \D^{+}(\CO_{X})$
and  $\F^{\bullet} \in \D^{+}(f^{-1}\CO_{X})$,
there is an isomorphism
\begin{equation}\label{adjoint-iso}
\Hom_{\D^{+}(\CO_{X})}(\E^{\bullet}, Rf_{\ast}\F^{\bullet})
\cong
\Hom_{\D^{+}(f^{-1}\CO_{X})}(f^{-1}\E^{\bullet}, \F^{\bullet})
\end{equation}
which is functorial for $\E^{\bullet}$ and $\F^{\bullet}$.
Due to the naturality of isomorphisms in \eqref{adjoint-iso},
the morphism above gives rise to a natural transformation
\begin{equation}\label{funct-morph}
\Id
\longrightarrow
Rf_{\ast}f^{-1}
\end{equation}
in $\D^{+}(\CO_{X})$ (cf. \cite[(2.6.16)]{KS94}).
In particular, for a distinguished triangle
$$
\xymatrix@C=0.5cm{
\E^{\bullet} \ar[r]& \F^{\bullet} \ar[r]& \mathcal{G}^{\bullet} \ar[r]& \E^{\bullet}[1],
}
$$
in $\D^{+}(\CO_{X})$,
we can construct a morphism of distinguished triangles
\begin{equation}\label{distinguish-commut}
\vcenter{
\xymatrix@C=0.4cm{
 \E^{\bullet} \ar[d]\ar[r]& \F^{\bullet} \ar[d]\ar[r]& \mathcal{G}^{\bullet}  \ar[d]\ar[r]& \E^{\bullet}[1] \ar[d]& \\
Rf_{\ast}f^{-1} \E^{\bullet}  \ar[r]& Rf_{\ast}f^{-1}\F^{\bullet}  \ar[r]&Rf_{\ast}f^{-1}\mathcal{G}^{\bullet} \ar[r]& Rf_{\ast}f^{-1} \E^{\bullet}[1]
}}
\end{equation}
in the derived category $\D^{+}(\CO_{X})$.
Since $R\Gamma(Y, -)=R\Gamma(X, Rf_{\ast}(-))$,
applying the derived functor $R\Gamma(X, -)$ to \eqref{distinguish-commut},
we get a commutative ladder of long exact sequences of hypercohomology groups
\begin{equation}\label{pre-comdiag-1}
\vcenter{
\xymatrix@C=0.4cm{
\cdots\ar[r]^{} & \mathbb{H}^{l}(X,\E^{\bullet})
\ar[d]_{} \ar[r]^{} & \mathbb{H}^{l}(X, \F^{\bullet})
\ar[d]_{} \ar[r]^{} & \mathbb{H}^{l}(X, \mathcal{G}^{\bullet})
\ar[d]_{} \ar[r]^{} & \mathbb{H}^{l+1}(X,\E^{\bullet}) \ar[d]_{} \ar[r]^{} & \cdots \\
\cdots\ar[r] & \mathbb{H}^{l}(Y, f^{-1}\E^{\bullet}) \ar[r]^{} &
\mathbb{H}^{l}(Y, f^{-1}\F^{\bullet}) \ar[r]^{} &
\mathbb{H}^{l}(Y, f^{-1}\mathcal{G}^{\bullet}) \ar[r]^{} &
\mathbb{H}^{l+1}(Y, f^{-1}\E^{\bullet}) \ar[r] & \cdots.}}
\end{equation}

\begin{rem}
Alternatively, one can construct the commutative diagrams \eqref{pre-comdiag-1} and \eqref{final-diagram} below by using the \v{C}ech (hyper)cohomology theory; for instance, see Serre's GAGA \cite[$\S$ 11]{Ser56}.
The advantage of the categorical construction above is to avoid some involved local calculations.
\end{rem}

In complex differential geometry,
each sheaf of holomorphic differential forms on a complex manifold admits a canonical resolution: the Dolbeault resolution.
As a result, by the Dolbeault Theorem,
the pullback of differential forms naturally induces a morphism of Dolbeault cohomology groups.
In algebraic geometry, the regular differential forms on a variety can be considered as the counterpart of differential forms in complex differential geometry.
By contrast, the sheaf of regular differential forms has no analog of the Dolbeault resolution, since the Zariski topology is coarser than the complex topology.
Naturally, given a morphism of smooth proper varieties one may wonder how to define a natural morphism of sheaves of regular differential forms and the induced morphism of their cohomology groups under this morphism.
The rest of this subsection is devoted to explain the induced morphism of sheaves of regular differential forms on smooth varieties (see \eqref{real-pullback}) and the induced morphisms of Hodge cohomology groups (see \eqref{induced-map}).

Assume that $\V$ is a locally free sheaf on $X$ and
let $f:Y\longrightarrow X$ be a morphism of smooth proper varieties.
For the sheaves $\Omega^{p}_{X}$ and $\Omega^{p}_{Y}$,
there is a natural commutative diagram
\begin{equation}\label{topolpullbackmorph}
\vcenter{
\xymatrix@C=0.5cm{
&f^{-1}\Omega_{X}^{p}(\V) \ar[d]^{} \ar[rd]^{\alpha} \\
& f^{-1}f_{\ast}\Omega_{Y}^{p}(f^{\ast}\V) \ar[r] &\Omega_{Y}^{p}(f^{\ast}\V). &
}}
\end{equation}
In the derived category $\D^{+}(\CO_{X})$,
combining \eqref{topolpullbackmorph}
with the functorial morphism \eqref{funct-morph} for $\Omega_{X}^{p}(\V)$ yields a natural morphism
\begin{equation}\label{real-pullback}
f^{\#}:\Omega_{X}^{p}(\V)
\longrightarrow
Rf_{\ast}f^{-1}\Omega_{X}^{p}(\V)
\stackrel{Rf_{\ast}(\alpha)}\longrightarrow
Rf_{\ast}\Omega_{Y}^{p}(f^{\ast}\V),
\end{equation}
which gives rise to a morphism of cohomology groups
\begin{equation}\label{induced-map}
f^{\#}: H^{q}(X, \Omega_{X}^{p}(\V))
\longrightarrow
H^{q}(Y, f^{-1}\Omega_{X}^{p}(\V))
\longrightarrow
H^{q}(Y, \Omega_{Y}^{p}(f^{\ast}\V)).
\end{equation}

\begin{rem}
If $f:Y\longrightarrow X$ is a morphism of smooth proper varieties and $\V$ is a locally free sheaf on $X$,
then one can define a similar natural morphism
$$
\beta:f^{\ast}\Omega^{p}_{X}(\V)\longrightarrow\Omega^{p}_{Y}(f^{\ast}\V)
$$
of $\CO_{Y}$-modules, and hence there is a corresponding morphism of Hodge cohomology groups.
\end{rem}

%=======================================================================

\subsection{Blow-ups}

Let $X$ be a smooth proper variety of dimension $n\geq 2$ and $\iota:Z\hookrightarrow X$ a smooth closed subvariety of codimension $c\geq 2$.
Denote by $\mathcal{I}_{Z}\subset \CO_{X}$ the coherent sheaf of ideals corresponding to $Z$.
Then the {\it blow-up $\tilde{X}$ of $X$ along $Z$} is defined to be
$$
\tilde{X}:=\mathrm{Bl}_{Z}X=\mathbf{Proj}\biggl(\bigoplus_{d\geq0} \mathcal{I}_{Z}^{d}\biggr),
$$
where $\mathcal{I}_{Z}^{d}$ is the $d$-th power of the ideal $\mathcal{I}$
and $\mathcal{I}^{0}:=\CO_{X}$.
By definition, there is a natural morphism $\pi:\tilde{X}\longrightarrow X$
such that
$\pi: \pi^{-1}(U)\longrightarrow U$
is biregularly isomorphic, where $U:=X-Z$.
We say that $\pi$ is the {\it blow-up morphism} and $E:=\pi^{-1}(Z)$ is the {\it exceptional divisor}.
Moreover, there is a commutative diagram
\begin{equation}\label{blow-up-diagram}
\vcenter{
\xymatrix@C=2cm{
E \ar[d]_{\rho} \ar@{^{(}->}[r]^{\tilde{\iota}} & \tilde{X}\ar[d]^{\pi}\\
 Z \ar@{^{(}->}[r]^{\iota} & X.}}
\end{equation}
Let $\mathcal{N}_{Z/X}$ be the normal bundle of rank $c$ of $Z$ in $X$.
In fact, the exceptional divisor $E$ is equal to the projectivization of $\mathcal{N}_{Z/X}$, i.e.,
$E=\mathbb{P}(\mathcal{N}_{Z/X})$.
Moreover, the morphism $\rho: E\longrightarrow Z$ is the projective fibration of $E$ over the center $Z$ and the following basic properties hold:
\begin{enumerate}
\item [(i)] the new variety $\tilde{X}$ is a smooth proper variety;
\item [(ii)] the blow-up morphism $\pi$ is projective;
\item [(iii)] $\CO_{\tilde{X}}(-E)$ is very ample relative to $\pi$
and $\CO_{\tilde{X}}(-E)|_{E}\cong \CO_{E}(1)$, where $\CO_{E}(1)$ is the Grothendieck line bundle of $E$.
\end{enumerate}

Now we consider the (higher) direct image of the structure sheaf under the blow-up morphism.
There hold $\pi^{\#}: \CO_{X}\stackrel{\simeq}\longrightarrow \pi_{\ast}\CO_{\tilde{X}}$
and $R^{i}\pi_{\ast}\CO_{\tilde{X}}=0$ for $i>0$ (cf. \cite[Chapter V, Proposition 3.4]{Har77}).
As a result, for a locally free sheaf of $\mathcal{O}_{X}$-modules $\V$,
there hold isomorphisms
$$
H^{q}(X, \V)\cong H^{q}(\tilde{X}, \pi^{\ast}\V)
$$
for any $q\geq0$.
In a more general setting,
if $f:\tilde{X}\longrightarrow X$ is a projective birational morphism of smooth varieties,
then we also have
$\CO_{X}\stackrel{\simeq}\longrightarrow f_{\ast}\CO_{\tilde{X}}$ and $R^{i}f_{\ast}\CO_{\tilde{X}}=0$ for any $i>0$ (see \cite{Hir64} and \cite{CR11}).

A natural problem comes to mind: {\it What about the sheaf of regular differential forms of degree $p$ such that $0<p<n$?}
In fact, one can show that there holds the isomorphism
$$
H^{0}(X, \Omega_{X}^{p}(\V))\cong H^{0}(\tilde{X}, \Omega_{\tilde{X}}^{p}(\pi^{\ast}\V))
$$
for any $0<p<n$.
Hence, the Hodge cohomologies of types $(p,0)$ and $(0,q)$ are invariant under the blow-up morphism.
However, for the general types the invariance of Hodge cohomology does not hold anymore.
The reason lies in the fact that the center has some contributions to the Hodge cohomology of the blowing up variety.

\begin{example}
Here is a simple example from \cite[Chapter V, Exercise 5.3]{Har77}.
Let $X$ be a smooth proper surface and
$\pi:\tilde{X}\longrightarrow X$ the blow-up of $X$ at a closed point $p\in X$.
For any locally free sheaf  $\V$ of rank $r$ on $X$, there exists a short exact sequence of sheaves
\begin{equation}\label{surface-case-}
\xymatrix@C=0.5cm{
0\ar[r] & \pi^{\ast}\Omega_{X}^{1}(\V)
\ar[r] &  \Omega_{\tilde{X}}^{1}(\pi^{\ast}\V)
\ar[r] &  \Omega_{\tilde{X}/X}^{1}(\pi^{\ast}\V)
\ar[r] & 0.
}
\end{equation}
First, for any $l\geq0$, we claim the following isomorphism as $k$-vector spaces:
\begin{equation}\label{pi*-iso}
H^{l}(\tilde{X},\pi^{\ast}\Omega^{1}_{X}(\V))\cong
H^{l}(X,\Omega^{1}_{X}(\V)).
\end{equation}
Consider the direct image (resp. higher direct images) of the sheaf $\pi^{\ast}\Omega^{1}_{X}(\V)$ along $\pi$.
From the projection formula we have
$
\pi_{\ast}(\pi^{\ast}\Omega^{1}_{X}(\V))
\cong \Omega^{1}_{X}(\V)
$
and
$
R^{i}\pi_{\ast}(\pi^{\ast}\Omega^{1}_{X}(\V))
= 0
$
for any $i\geq1$.
Using the Leray spectral sequence for $\pi_{\ast}(\pi^{\ast}\Omega^{1}_{X}(\V))$, it is a direct consequence that \eqref{pi*-iso} holds.
As $\tilde{X}$ is the pointed blow-up of $X$, we have $\Omega_{\tilde{X}/X}^{1}\cong  \tilde{\iota}_{\ast}\Omega_{E}^{1}$.
Observe that
$\Omega_{\tilde{X}/X}^{1}(\pi^{\ast}\V)\cong  \tilde{\iota}_{\ast}(\Omega_{E}^{1})^{\oplus r}$
from the projection formula.
Hence, we get
\begin{equation}\label{1-1}
H^{l}(\tilde{X},\Omega_{\tilde{X}/X}^{1}(\pi^{\ast}\V))\cong
H^{l}(\tilde{X},\tilde{\iota}_{\ast}(\Omega_{E}^{1})^{\oplus r})
\cong
H^{l}(\mathbb{P}^1, (\Omega_{\mathbb{P}^1}^{1})^{\oplus r})
=
 \left\lbrace
           \begin{array}{c l}
             k^{\oplus r}, & \text{$l=1$};\\
             0, & \text{$l\neq1$},
           \end{array}
         \right.
\end{equation}
since $E\cong \mathbb{P}^1$.
Consider the long exact sequence of sheaf cohomology groups from \eqref{surface-case-}.
By \eqref{pi*-iso}, and \eqref{1-1}, we obtain
a short exact sequence of $k$-vector spaces
$$
\xymatrix{
  0 \ar[r] & H^{1}(X,\Omega^{1}_{X}(\V))
  \ar[r]^{} & H^{1}(\tilde{X},\Omega^{1}_{\tilde{X}}(\pi^{\ast}\V)) \ar[r]^{} & k^{\oplus r} \ar[r] & 0 }
$$
and hence the isomorphism
$$
H^{1}(\tilde{X}, \Omega_{\tilde{X}}^{1}(\pi^{\ast}\V))\cong H^{1}(X, \Omega_{X}^{1}(\V))\oplus k^{\oplus r}.
$$
\end{example}

\begin{rem}
Suppose that $X$ is a smooth projective surface over a field $K$ (not necessarily algebraically closed).
Consider the blow-up $\pi: \tilde{X}\longrightarrow X$ of $X$ at a closed point $x\in X$.
Similarly, for any locally free sheaf $\V$ of rank $r$ on $X$,
one still has
$$
H^{1}(\tilde{X}, \Omega_{\tilde{X}}^{1}(\pi^{\ast}\V))\cong H^{1}(X, \Omega_{X}^{1}(\V))\oplus K(x)^{\oplus r},
$$
where $K(x):=\CO_{X,\, x}/\mathfrak{m}_{x}$ is the residue field of $x$ on $X$.
\end{rem}

In general, from \eqref{blow-up-diagram} we have the following important lemma for the proof of the Hodge blow-up formula later.

\begin{lem}\label{iso1-2-3}
For any $0\leq p\leq n$, we have:
\begin{itemize}
  \item [(i)]
  $\pi^{\#}: \Omega^{p}_{X}\stackrel{\simeq}\longrightarrow \pi_{\ast}\Omega^{p}_{\tilde{X}}$;
  \item [(ii)]
  $\rho^{\#}: \Omega^{p}_{Z}\stackrel{\simeq}\longrightarrow \rho_{\ast}\Omega^{p}_{E}$;
  \item [(iii)]\label{key-iso}
  $\tilde{\iota}^{\#}: R^{i}\pi_{\ast}\Omega^{p}_{\tilde{X}} \stackrel{\simeq}\longrightarrow \iota_{\ast}R^{i}\rho_{\ast}\Omega^{p}_{E}$ for any $i\geq 1$.
\end{itemize}
\end{lem}

This third isomorphism in Lemma \ref{iso1-2-3} is first addressed by Gros \cite[Chapter IV, Theorem 1.2.1]{Gros85} for $i>1$ over an arbitrary base scheme.
In their paper \cite[(3.3) Proposition]{GNA02}, Guill\'{e}n--Navarro Aznar improved it to $i\geq1$ in characteristic zero.
It came as a surprise to us that it still holds in positive characteristic.
The main reason why this is possible is that the proof of this isomorphism is essentially based upon some principles from sheaf cohomology theory which are independent on the ground field. For reader's convenience, we present a complete proof here but do not claim any originality.

\begin{proof}[Proof of Lemma \ref{key-iso}]
Based on the Algebraic Hartogs Theorem \cite[Chapter II, Proposition 6.3A]{Har77}
and local trivialization of projective bundles,
the proofs of the assertions (i) and (ii) are quite similar to \cite[Lemma 4.1.(i)-(ii)]{RYY19b}.

Now we consider the assertion (iii).
First, note that there exist two standard short exact sequences associated with the exceptional divisor $E$ in $\tilde{X}$:
\begin{equation}\label{structure-seq}
\xymatrix@C=0.5cm{
  0 \ar[r] &  \CO_{\tilde{X}}(-E)  \ar[r] &  \CO_{\tilde{X}} \ar[r] & \tilde{\iota}_{\ast}\CO_{E}  \ar[r] & 0 }
\end{equation}
and
\begin{equation}\label{normal-seq}
\xymatrix@C=0.5cm{
  0 \ar[r] & \mathcal{N}_{E/\tilde{X}}^{\vee} \ar[r] & \tilde{\iota}^{\ast}\Omega_{\tilde{X}} \ar[r] & \Omega_{E} \ar[r] & 0, }
\end{equation}
where $\mathcal{N}_{E/\tilde{X}}^{\vee}\cong \CO_{E}(1)\cong \tilde{\iota}^{\ast}\CO_{\tilde{X}}(-E)$.
 Write $\CO_{\tilde{X}}(1):=\CO_{\tilde{X}}(-E)$ and thus $\CO_{E}(1)=\tilde{\iota}^{\ast}\CO_{\tilde{X}}(1)$.
Twisting \eqref{structure-seq} with $\Omega^{p}_{\tilde{X}}\otimes \CO_{\tilde{X}}(m)$ gives rise to a short exact sequence
\begin{equation}\label{twist-structure-seq}
\xymatrix@C=0.5cm{
  0 \ar[r] & \Omega^{p}_{\tilde{X}}
  \otimes \CO_{\tilde{X}}(m+1)  \ar[r] &  \Omega^{p}_{\tilde{X}} \otimes
  \CO_{\tilde{X}}(m) \ar[r] &
  \Omega^{p}_{\tilde{X}}
  \otimes \tilde{\iota}_{\ast} \CO_{E}(m)  \ar[r] & 0. }
\end{equation}
Since $\mathcal{N}_{E/\tilde{X}}$ is an invertible sheaf,
taking $p$-th exterior wedge of \eqref{normal-seq} and then twisting it with $\CO_{E}(m)$,
one gets another short exact sequence
\begin{equation}\label{wedge-normal-seq}
\xymatrix@C=0.5cm{
  0 \ar[r] &  \Omega^{p-1}_{E}
  \otimes \CO_{E}(m+1) \ar[r] & \tilde{\iota}^{\ast}\Omega^{p}_{\tilde{X}}
  \otimes \CO_{E}(m)  \ar[r] & \Omega^{p}_{E}
  \otimes  \CO_{E}(m)\ar[r] & 0.}
\end{equation}
Due to the projection formula, for any $m\geq 0$, one has

\begin{eqnarray}\label{relat-direct-im}
R^{i}\pi_{\ast}(\Omega^{p}_{\tilde{X}}\otimes \tilde{\iota}_{\ast} \CO_{E}(m))
&\stackrel{\simeq}\longrightarrow&
R^{i}\pi_{\ast}(\tilde{\iota}_{\ast} (\tilde{\iota}^{\ast} \Omega^{p}_{\tilde{X}}
\otimes \CO_{E}(m))) \nonumber\\
&\cong& R^{i}(\pi \circ \tilde{\iota})_{\ast} (\tilde{\iota}^{\ast}\Omega^{p}_{\tilde{X}}
\otimes \CO_{E}(m)) \nonumber\\
&\cong& R^{i}(\iota \circ \rho)_{\ast} (\tilde{\iota}^{\ast}\Omega^{p}_{\tilde{X}}
\otimes \CO_{E}(m)) \nonumber \\
&\cong&\iota_{\ast}R^{i}\rho_{\ast} (\tilde{\iota}^{\ast}\Omega^{p}_{\tilde{X}}
\otimes \CO_{E}(m)).
\end{eqnarray}
The second isomorphism in \eqref{relat-direct-im} comes from the fact
$R^{j}\tilde{\iota}_{\ast} (\tilde{\iota}^{\ast} \Omega^{p}_{\tilde{X}}
\otimes \CO_{E}(m))=0$ for any $j>0$ since $\tilde{\iota}$ is a closed immersion, and the Grothendieck spectral sequence
$$
E^{i,j}_{2}=
R^{i}\pi_{\ast}(R^{j}\tilde{\iota}_{\ast} (\tilde{\iota}^{\ast} \Omega^{p}_{\tilde{X}}
\otimes \CO_{E}(m)))\Longrightarrow
R^{i+j}(\pi \circ \tilde{\iota})_{\ast} (\tilde{\iota}^{\ast}\Omega^{p}_{\tilde{X}}
\otimes \CO_{E}(m)).
$$

\begin{claim}\label{highvani-B}
For any $i\geq 1$ and $m\geq 1$,
$R^{i}\pi_{\ast}(\Omega^{p}_{\tilde{X}}\otimes \tilde{\iota}_{\ast} \CO_{E}(m))=0$.
\end{claim}

\begin{proof}
From the isomorphism \eqref{relat-direct-im},
it suffices to show
$R^{i}\rho_{\ast}(\tilde{\iota}^{\ast}\Omega^{p}_{\tilde{X}} \otimes \CO_{E}(m))=0$
for any $m\geq 1$ and $i\geq 1$.
Our first goal is to show
$
R^{i}\rho_\ast(\Omega^{p}_{E}\otimes\CO_{E}(m))=0
$
and then the vanishing of the term
$R^{i}\rho_{\ast}(\tilde{\iota}^{\ast}\Omega^{p}_{\tilde{X}} \otimes \CO_{E}(m))$
follows from the exactness of the long exact sequence of the higher direct images for \eqref{wedge-normal-seq}.
Actually, this is a local problem over the center $Z$.
Note that $E$ is the projectivization of $\mathcal{N}_{Z/X}$ which admits local triviality.
Without loss of generality, we assume that $Z=\mathrm{Spec}\,A$ is a smooth variety
and $E=\mathrm{Spec}\,A \times \mathbb{P}_{k}^{c-1}$ is a product space.
By the K\"{u}nneth formula, we have

\begin{eqnarray*}
H^{i}(E,\Omega^{p}_{E}\otimes\CO_{E}(m))
&\cong&
\bigoplus_{{0\leq r\leq p \atop{0\leq s\leq i}}} H^{i-s}(\mathrm{Spec}\,A,\Omega_{\mathrm{Spec}\,A}^{r})
\otimes H^{s}(\mathbb{P}^{c-1}, \Omega_{\mathbb{P}^{c-1}}^{p-r}\otimes \CO_{\mathbb{P}^{c-1}}(m))\\
&\cong&
\bigoplus_{0\leq r\leq p} H^{0}(\mathrm{Spec}\,A,\Omega_{\mathrm{Spec}\,A}^{r})
\otimes
H^{i}(\mathbb{P}^{c-1}, \Omega_{\mathbb{P}^{c-1}}^{p-r}
\otimes \CO_{\mathbb{P}^{c-1}}(m))
\end{eqnarray*}
since there holds $H^{i-s}(\mathrm{Spec}\, A, \Omega_{\mathrm{Spec}\,A}^{r})=0$ when $i-s>0$.
Moreover, the Bott formula implies that the cohomology group
$H^{i}(\mathbb{P}^{c-1}, \Omega_{\mathbb{P}^{c-1}}^{p-r}\otimes \CO_{\mathbb{P}^{c-1}}(m))$ vanishes for any $i\geq 1$ and $m\geq 1$, see \cite[Proposition 14.4]{Bot57} or \cite[Theorem 4.5]{Hua01}.
This implies
$H^{i}(E, \Omega^{p}_{E} \otimes  \CO_{E}(m))=0$ for any
$i\geq 1$ and $m\geq 1$ and therefore, from definition, we are led to the conclusion that the higher direct images
$
R^{i}\rho_\ast(\Omega^{p}_{E}\otimes\CO_{E}(m))
$
vanish.
\end{proof}

Consider the long exact sequence of the higher direct images of \eqref{wedge-normal-seq} for $m=0$.
Because
$R^{i}\rho_{\ast}(\Omega^{p}_{E} \otimes  \CO_{E}(1))=0$
for every $i\geq 1$, the exactness of the long exact sequence implies the isomorphism
$R^{i}\rho_{\ast}\tilde{\iota}^{\ast}\Omega_{\tilde{X}}^{p} \stackrel{\simeq}\longrightarrow
R^{i}\rho_{\ast}\Omega^{p}_{E}$.
Set $m=0$ in \eqref{relat-direct-im} and therefore we get
\begin{equation*}
R^{i}\pi_{\ast}(\Omega_{\tilde{X}}^{p}  \otimes \tilde{\iota}_{\ast} \CO_{E})
\stackrel{\simeq}\longrightarrow
\iota_{\ast}R^{i}\rho_{\ast} \tilde{\iota}^{\ast}\Omega_{\tilde{X}}^{p}
\stackrel{\simeq}\longrightarrow
\iota_{\ast}R^{i}\rho_{\ast} \Omega_{E}^{p}.
\end{equation*}
Consequently,  to complete the proof, it is sufficient to show
\begin{equation}\label{eq-equal}
R^{i}\pi_{\ast}\Omega^{p}_{\tilde{X}}
\stackrel{\simeq}\longrightarrow
R^{i}\pi_{\ast}(\Omega_{\tilde{X}}^{p}  \otimes \tilde{\iota}_{\ast} \CO_{E}).
\end{equation}
Now consider the long exact sequence of the higher direct images for \eqref{twist-structure-seq}.
Thanks to Claim \ref{highvani-B}, the morphism
\begin{equation}\label{vanishsurj}
R^{i}\pi_{\ast}(\Omega_{\tilde{X}}^{p} \otimes \CO_{\tilde{X}}(m+1))
\longrightarrow
R^{i}\pi_{\ast}\bigl
(\Omega_{\tilde{X}}^{p} \otimes \CO_{\tilde{X}}(m))
\end{equation}
is surjective for any $m\geq 1$.
Observe that $\CO_{\tilde{X}}(1)$ is very ample with respect to the projective morphism $\pi$.
According to the relative Serre vanishing theorem \cite[Chapter III, Theorem 8.8 (c)]{Har77},
there is a positive integer $l_{0}$ such that for any $l\geq l_0$,
\begin{equation}\label{serre-zero}
R^{i}\pi_{\ast}(\Omega_{\tilde{X}}^{p} \otimes \CO_{\tilde{X}}(l))
=0.
\end{equation}
From \eqref{vanishsurj} and \eqref{serre-zero},
by induction we have
\begin{equation}\label{eq-equal2}
R^{i}\pi_{\ast}\bigl
(\Omega_{\tilde{X}}^{p} \otimes \CO_{\tilde{X}}(1))=0.
\end{equation}
Finally, let us turn back to the long exact sequence of the higher direct images of \eqref{twist-structure-seq} for $m=0$.
It follows from \eqref{eq-equal2} that the isomorphism \eqref{eq-equal} holds and this completes the proof of Lemma \ref{key-iso}.
\end{proof}

%=====================================================================
\section{Relative Hodge sheaves}

In this section, we introduce the notion of relative Hodge sheaves and prove the isomorphism \eqref{icrhs} in Theorem \ref{mainthm}.

Let $X$ be a smooth proper variety of dimension $n$ and $\iota:Z\hookrightarrow X$ a smooth closed subvariety.
From definition of closed subvariety,
there is a natural {\it surjective} morphism
\begin{equation}\label{closed-embedding}
\iota^{\#}: \CO_{X}\longrightarrow \iota_{\ast}\CO_{Z}
\end{equation}
and the kernel of $\iota^{\#}$ is the coherent sheaf of ideals $\mathcal{I}_{Z}$ of $Z$ in $X$.
As a consequence, there exists a natural short exact sequence of coherent $\CO_{X}$-modules
\begin{equation*}
\xymatrix@C=0.5cm{
  0 \ar[r] & \mathcal{I}_{Z} \ar[r]^{} & \CO_{X}\ar[r]^{\iota^{\#}} & \iota_{\ast}\CO_{Z} \ar[r] & 0.}
\end{equation*}
In fact, the notion of relative Hodge sheaves is a generalization of the ideal sheaf $\mathcal{I}_{Z}$ above.
Generally, we consider the sheaves of regular differential $p$-forms over $X$ and $Z$.

\begin{defn}\label{rhs}
For any $0\leq p\leq n$, the {\it $p$-th relative Hodge sheaf} associated to the pair $(X, Z)$ is defined to be
the kernel sheaf
\begin{equation}\label{rel-Hodge-sh}
\K_{Z}^{p}
:=
\ker\big(\Omega_{X}^{p} \stackrel{\iota^{\#}}\longrightarrow  \iota_{\ast}\Omega_{Z}^{p}\big),
\end{equation}
where $\iota^{\#}$ is the natural pullback of regular differential $p$-forms.
\end{defn}

Note that $\K_{Z}^{0}$ is the ideal sheaf $\mathcal{I}_{Z}$ and $\K_{Z}^{p}=\Omega_{X}^{p}$ if $p>\mathrm{dim}\, Z$.
Moreover, we have the following lemma.

\begin{lem}\label{short-exact}
For any $0\leq p\leq n$,
there exists a short exact sequence
\begin{equation}\label{key-short-exact}
\xymatrix@C=0.5cm{
  0 \ar[r] & \K^{p}_{Z} \ar[r]^{} & \Omega_{X}^{p} \ar[r]^{\iota^{\#}} & \iota_{\ast}\Omega_{Z}^{p} \ar[r] & 0}
\end{equation}
of $\CO_{X}$-modules.
\end{lem}

\begin{proof}
It suffices to show that the morphism $\iota^{\#}$ in \eqref{rel-Hodge-sh} is surjective.
In fact, this is a direct consequence of local calculation and the surjectivity of \eqref{closed-embedding}.
\end{proof}

From now on, we assume that $X$ is a smooth proper variety over $k$ and $\iota:Z\hookrightarrow X$ is a smooth closed subvariety of codimension $c\geq 2$.
Let $\pi:\tilde{X}\rightarrow X$ be the blow-up of $X$ along $Z$ and $\tilde{\iota}:E\hookrightarrow\tilde{X}$ the exceptional divisor.
Set $\rho=\pi|_{E}:E\rightarrow Z$.
Analogous to \eqref{key-short-exact},
there is a short exact sequence of $\CO_{\tilde{X}}$-modules associated with the pair $(\tilde{X}, E)$:
\begin{equation}\label{key-short-exact-E}
\xymatrix@C=0.5cm{
  0 \ar[r] & \K^{p}_{E} \ar[r]^{} & \Omega_{\tilde{X}}^{p} \ar[r]^{\iota^{\#}} & \tilde{\iota}_{\ast}\Omega_{E}^{p} \ar[r] & 0.}
\end{equation}
Observe that $\tilde{X}-E$ is isomorphic to $X-Z$.
Intuitively, the relative Hodge sheaves $\K^{p}_{E}$ and $\K^{p}_{Z}$ are \lq\lq \emph{geometrically}" dependent on $\tilde{X}-E$ and $X-Z$, respectively.
This implies that $\K^{p}_{E}$ should be the \lq\lq \emph{same}" as $\K^{p}_{Z}$ in some sense.
The following lemma explains such an \lq\lq \emph{equality}".

\begin{lem}\label{key-lem}
Let $\V$ be a locally free sheaf over $X$ and set $\tilde{\V}:=\pi^{\ast}\V$.
Then we have an isomorphism
$\pi^{\#}: \K_{Z}^{p}(\V)
\stackrel{\simeq}\longrightarrow
\pi_{\ast}\K_{E}^{p}(\tilde{\V})$
and $R^{i}\pi_{\ast}\K_{E}^{p}(\tilde{\V})=0$ for any $i\geq 1$.
\end{lem}

\begin{proof}
Note that $\iota$ and $\tilde{\iota}$ are closed inclusions.
Due to Lemma \ref{key-iso} and the commutativity of the blow-up diagram
$\pi\circ \tilde{\iota}=\iota\circ\rho$,
we obtain the following canonical isomorphisms
\begin{equation}\label{import-isos-on-E}
R^{i}\pi_{\ast}\Omega^{p}_{\tilde{X}}
\stackrel{\simeq}\longrightarrow
\iota_{\ast}R^{i}\rho_{\ast}\Omega^{p}_{E}
\cong  R^{i}(\iota\circ\rho)_{\ast}\Omega^{p}_{E}
\cong R^{i}(\pi\circ\tilde{\iota})_{\ast}\Omega^{p}_{E}
\cong R^{i}\pi_{\ast}\tilde{\iota}_{\ast}\Omega^{p}_{E}
\end{equation}
for each $i\geq 1$.
As a consequence,  applying the projection formula to \eqref{import-isos-on-E} yields an isomorphism
\begin{equation}\label{import-isos-on-E-1}
R^{i}\pi_{\ast}\Omega^{p}_{\tilde{X}}(\tilde{\V})
\stackrel{\simeq}\longrightarrow
R^{i}\pi_{\ast}\tilde{\iota}_{\ast}\Omega^{p}_{E}
(\tilde{\iota}^{\ast}\tilde{\V})
\end{equation}
for each $i\geq 1$.
Via tensoring \eqref{key-short-exact-E} with the locally free sheaf $\tilde{\V}$,
we get a short exact sequence
\begin{equation}\label{twist-rel-Dol-sequ-holom-E}
\xymatrix@C=0.5cm{
  0 \ar[r] & \K_{E}^{p}(\tilde{\V})  \ar[r]^{} & \Omega_{\tilde{X}}^{p}(\tilde{\V})   \ar[r]^{\tilde{\iota}^{\#}\;\;\;} & \tilde{\iota}_{\ast}\Omega_{E}^{p}(\tilde{\iota}^{\ast}\tilde{\V}) \ar[r] & 0.}
\end{equation}
Consider the higher direct images of \eqref{twist-rel-Dol-sequ-holom-E} along $\pi$.
Then there is a long exact sequence
\begin{equation}\label{l-h-d}
\begin{tikzcd}
  0\rar & \pi_{\ast}\K_{E}^{p}(\tilde{\V})  \rar &  \pi_{\ast}\Omega^{p}_{\tilde{X}}(\tilde{\V})  \rar & \pi_{\ast}\tilde{\iota}_{\ast}
  \Omega^{p}_{E}(\tilde{\iota}^{\ast}\tilde{\V}) \ar[out=-23, in=150]{dll} \\
  & R^{1}\pi_{\ast}\K_{E}^{p}(\tilde{\V})  \rar & R^{1}\pi_{\ast}\Omega^{p}_{\tilde{X}}(\tilde{\V}) \rar &R^{1}\pi_{\ast}\tilde{\iota}_{\ast}\Omega^{p}_{E}(\tilde{\iota}^{\ast}\tilde{V})\longrightarrow \cdots.
\end{tikzcd}
\end{equation}
Combining \eqref{import-isos-on-E-1} with the exactness of the sequence \eqref{l-h-d},
we get $R^{i}\pi_{\ast}\K_{E}^{p}(\tilde{\V})=0$ for any $i\geq 2$.
Now we claim $\pi^{\#}: \K_{Z}^{p}(\V)\stackrel{\simeq}\longrightarrow \pi_{\ast}\K_{E}^{p} (\tilde{\V})$ and $R^{1}\pi_{\ast}\K_{E}^{p}(\tilde{\V})=0$.
By Lemma \ref{iso1-2-3},
the blow-up diagram \eqref{blow-up-diagram} gives a commutative diagram
\begin{equation}\label{kershisodiag}
\vcenter{
\xymatrix@C=0.4cm{
0 \ar[r]^{} & \K_{Z}^{p}(\V)
\ar[d]_{\pi^{\#}}^{} \ar[r]^{} &
\Omega_{X}^{p}(\V)
\ar[d]_{\pi^{\#}}^{\cong} \ar[r]^{\iota^{\#}} &
\iota_{\ast}\Omega_{Z}^{p}(\iota^{\ast}\V)
\ar[d]_{{\rho}^{\#}}^{\cong} \ar[r]^{} & 0 \\
0 \ar[r] &
\pi_{\ast}\K_{E}^{p}(\tilde{\V})
\ar[r]^{} &
\pi_{\ast}\Omega_{\tilde{X}}^{p}(\tilde{\V})
\ar[r]^{\tilde{\iota}^{\#}} &
\pi_{\ast}\tilde{\iota}_{\ast}\Omega_{E}^{p}
(\tilde{\iota}^{\ast}\tilde{\V}).
}}
\end{equation}
The commutativity of \eqref{kershisodiag} implies that the morphism $\tilde{\iota}^{\#}$ in \eqref{kershisodiag} is surjective and therefore there holds the isomorphism
$\pi^{\#}: \K_{Z}^{p}(\V)\stackrel{\simeq}\longrightarrow \pi_{\ast}\K_{E}^{p} (\tilde{\V})$.
Also, from the exactness of \eqref{l-h-d} and the surjectivity of $\tilde{\iota}^{\#}$, we get $R^{1}\pi_{\ast}\K_{E}^{p}(\tilde{\V})=0$ and
this completes the proof.
\end{proof}

\begin{rem}
Consider the $0$-th relative Hodge sheaf, i.e., the ideal sheaf $\mathcal{I}_{Z}$.
For any $a\geq 0$ and $i>0$, one can show that the following results hold:
$$
\pi_{\ast}\CO_{\tilde{X}}(-aE)\cong \mathcal{I}_{Z}^{a}\;\;
\textrm{and}
\;\;
R^{i}\pi_{\ast}\CO_{\tilde{X}}(-aE)=0,
$$
where $\mathcal{I}_{Z}^{a}$ is $a$-th power of the ideal sheaf $\mathcal{I}_{Z}$ (cf. \cite[Lemma 4.3.16]{Laz04}).
In general, without the assumption of smoothness for $X$,
one has to be content with large values of $a$  (cf. \cite[Lemma 5.4.24]{Laz04}).
\end{rem}

The following theorem about the cohomology of relative Hodge sheaves is crucial for the proof in the next section.
\begin{thm}[{ = Theorem \ref{mainthm}, \eqref{icrhs}}]\label{key-tech-prop}
For any integer $l\geq0$,
the induced morphism
\begin{equation*}
\pi^{\#}:
H^{l}(X, \K_{Z}^{p}(\V))
\longrightarrow
H^{l}(\tilde{X}, \pi^{-1}\K_{Z}^{p}(\V))
\longrightarrow
H^{l}(\tilde{X}, \K_{E}^{p}(\tilde{\V}))
\end{equation*}
is an isomorphism.
\end{thm}

\begin{proof}
Similar to \eqref{topolpullbackmorph},
there is a natural composition morphism,
\begin{equation}\label{kersh-diagram}
\vcenter{
\xymatrix@C=0.5cm{
&\pi^{-1}\K_{Z}^{p}(\V) \ar[d]^{} \ar[rd]^{\alpha} \\
& \pi^{-1}\pi_{\ast}\K_{E}^{p}(\tilde{\V}) \ar[r] & \K_{E}^{p}(\tilde{\V}) &
}}
\end{equation}
of $f^{-1}\CO_{X}$-modules.
Likewise, we have the induced morphism as \eqref{real-pullback}
$$
\K_{Z}^{p}(\V)
\longrightarrow
R\pi_{\ast}\pi^{-1}\K_{Z}^{p}(\V)
\longrightarrow
R\pi_{\ast}\K_{E}^{p}(\tilde{\V})
$$
in the derived category $\D^{+}(\CO_{X})$ and hence the induced morphism of cohomology groups
$$
\pi^{\#}: H^{l}(X, \K_{Z}^{p}(\V))
\longrightarrow
H^{l}(\tilde{X}, \pi^{-1}\K_{Z}^{p}(\V))
\stackrel{\alpha}\longrightarrow
H^{l}(\tilde{X}, \K_{E}^{p}(\tilde{\V})).
$$
Applying \eqref{funct-morph} to $\pi_{\ast}\K_{E}^{p}(\tilde{\V})$ with respect to $\pi$,
one obtains a natural morphism
\begin{equation*}\label{real-pullback-E}
\pi_{\ast}\K_{E}^{p}(\tilde{\V})
\longrightarrow
R\pi_{\ast}\pi^{-1}\pi_{\ast}\K_{E}^{p}(\tilde{\V})
\end{equation*}
in $\D^{+}(\CO_{X})$.
The isomorphism $\pi^{\#}: \K_{Z}^{p}(\V) \stackrel{\simeq}\longrightarrow \pi_{\ast}\K_{E}^{p}(\tilde{\V}) $ in Lemma \ref{key-lem},
the functorial property of \eqref{funct-morph} and the commutativity of \eqref{kersh-diagram} yield a commutative diagram with vertical isomorphisms
\begin{equation}\label{kershcohom-diagram}
\vcenter{
\xymatrix@C=0.4cm{
&H^{l}(X, \K_{Z}^{p}(\V)) \ar[d]_{}^{\cong} \ar[r]  &
H^{l}(\tilde{X}, \pi^{-1}\K_{Z}^{p}(\V)) \ar[d]_{}^{\cong} \ar[rd] \\
& H^{l}(X, \pi_{\ast}\K_{E}^{p}(\tilde{\V})) \ar[r] & H^{l}(\tilde{X}, \pi^{-1}\pi_{\ast}\K_{E}^{p}(\tilde{\V})) \ar[r] & H^{l}(\tilde{X}, \K_{E}^{p}(\tilde{\V})). &
}}
\end{equation}

To show that $\pi^{\#}$ is an isomorphism,
we consider the Leray spectral sequence of $\K_{E}^{p}(\tilde{\V})$
under the blow-up morphism $\pi:\tilde{X}\longrightarrow X$.
Then there exists a spectral sequence $\{E_{r}\}$ with the $E_{2}$-terms
$$
E_{2}^{s,t}
=
H^{s}(X, R^{t}\pi_{\ast}\K_{E}^{p}(\tilde{\V})),
$$
converging to a limit term of $E_{\infty}^{s, l-s}$
which is a graded piece of the graded vector space $H^{l}(\tilde{X},\K_{E}^{p}(\tilde{\V}))$ with respect to a given filtration.
Moreover, from a standard result in spectral sequence theory \cite[(13.8) Theorem of Chapter IV]{Dem12}, the edge morphism
$$
H^{l}(X,\pi_{\ast}\K_{E}^{p}(\tilde{\V}))
\twoheadrightarrow E^{l,0}_{\infty}
\hookrightarrow H^{l}(\tilde{X},\K_{E}^{p}(\tilde{\V}))
$$
is indeed the composition morphism
\begin{equation}\label{edge-map}
H^{l}(X, \pi_{\ast}\K_{E}^{p}(\tilde{\V}))
\longrightarrow
H^{l}(\tilde{X}, \pi^{-1}\pi_{\ast}\K_{E}^{p}(\tilde{\V}))
\longrightarrow
H^{l}(\tilde{X}, \K_{E}^{p}(\tilde{\V})).
\end{equation}
Again by Lemma \ref{key-lem}, we have
$R^{i}\pi_{\ast}\K_{E}^{p}(\tilde{\V})=0$ for $i\geq 1$ and
hence $E^{s, t}_{2}=0$ for any $t\geq1$.
It follows that the edge morphism \eqref{edge-map} is an isomorphism.
As a result, from the commutative diagram \eqref{kershcohom-diagram}
we obtain that the morphism
\begin{equation*}
\pi^{\#}:
H^{l}(X, \K_{Z}^{p}(\V))
\longrightarrow
H^{l}(\tilde{X}, \pi^{-1}\K_{Z}^{p}(\V))
\longrightarrow
H^{l}(\tilde{X}, \K_{E}^{p}(\tilde{\V}))
\end{equation*}
is an isomorphism and the proof is now complete.
\end{proof}

\begin{rem}
Abstractly, as a direct consequence of Lemma \ref{key-lem} and the degeneracy of the Leray spectral sequence at $E_{2}$, we get an isomorphism between $H^{l}(X, \K_{Z}^{p}(\V))$ and $H^{l}(\tilde{X}, \K_{E}^{p}(\tilde{\V}))$ as vector spaces over $k$.
The main reason why we use the argument in Theorem \ref{key-tech-prop} is that the abstract isomorphism above is not canonical.
However, in our proof of the Hodge blow-up formula below we need a \emph{canonical} isomorphism from
$H^{l}(X, \K_{Z}^{p}(\V))$ to
$H^{l}(\tilde{X}, \K_{E}^{p}(\tilde{\V}))$
which is induced by the blow-up morphism $\pi$.
\end{rem}

%==================================================================

\section{Blow-up formula of Hodge cohomology}

The purpose of this section is to explain the sheaf-theoretic proof of blow-up formula \eqref{bundle-blowup-formula} of Hodge cohomology in Theorem \ref{mainthm}.

Suppose that $\V$ is a locally free sheaf over $X$.
We will show that the blow-up diagram \eqref{blow-up-diagram} yields a commutative diagram of Hodge cohomology groups:
\begin{equation}\label{coh-blow-up-diagram}
\vcenter{
\xymatrix@C=2cm{
 H^{q}(X,\Omega^{p}_{X}(\V))\ar[d]_{\pi^{\#}} \ar[r]^{\iota^{\#}} & H^{q}(Z,\Omega^{p}_{Z}(\iota^\ast\V))\ar[d]^{\rho^{\#}}\\
 H^{q}(\tilde{X},\Omega^{p}_{\tilde{X}}(\tilde{\V})) \ar[r]^{\tilde{\iota}^{\#}} & H^{q}(E,\Omega^{p}_{E}(\tilde{\iota}^\ast\tilde{\V})),}}
\end{equation}
where $\tilde{\V}=\pi^\ast\V$.
To obtain the Hodge blow-up formula, one needs to show that $\pi^{\#}$ and $\rho^{\#}$ in \eqref{coh-blow-up-diagram} are \emph{injective} and then verify that the morphism $\tilde{\iota}^{\#}$ induces an isomorphism from the co-kernel of $\pi^{\#}$ to the co-kernel of $\rho^{\#}$.
Finally, to describe the term $\mathrm{coker}\,(\rho^{\#})$ explicitly, we establish the projective bundle formula for Hodge cohomology.
The trick of the proof is to plug the square \eqref{coh-blow-up-diagram} into a commutative diagram of long exact sequences containing the terms of sheaf cohomology of relative Hodge sheaves.
Then we can apply some results in homological algebra to complete the argument.

We divide the proof into three steps.
\paragraph{\textbf{Step 1}}
Consider the pair $(X,Z)$.
According to Lemma \ref{short-exact}, there is a natural short exact sequence of sheaves over $X$:
$$
\xymatrix@C=0.5cm{
  0 \ar[r] & \K^{p}_{Z} \ar[r]^{} & \Omega_{X}^{p} \ar[r]^{\iota^{\#}} & \iota_{\ast}\Omega_{Z}^{p} \ar[r] & 0.}
$$
Twisting the sequence above with $\V$ and using the projection formula gives rise to a short exact sequence
\begin{equation}\label{twist-exactseq-X}
\xymatrix@C=0.5cm{
  0 \ar[r] & \K_{Z}^{p}(\V) \ar[r]^{} & \Omega_{X}^{p}(\V) \ar[r]^{\iota^{\#}} & \iota_{\ast}\Omega_{Z}^{p}(\iota^{\ast}\V) \ar[r] & 0.}
\end{equation}
So, following the construction \eqref{pre-comdiag-1}, we get a commutative ladder of long exact sequences from \eqref{twist-exactseq-X}:
\begin{equation}\label{long-sequ1}
\vcenter{
\xymatrix@C=0.3cm{
\cdots\ar[r]^{} & H^{q}(X,\K_{Z}^{p}(\V)) \ar[d]_{} \ar[r]^{} &H^{q}(X, \Omega_{X}^{p}(\V)) \ar[d]_{} \ar[r]^{} & H^{q}(X, \iota_{\ast}\Omega_{Z}^{p}(\iota^{\ast}\V))\ar[d]_{} \ar[r]^{} & H^{q+1}(X,\K_{Z}^{p}(\V))\ar[d]_{} \ar[r]^{} & \cdots \\
\cdots\ar[r] & H^{q}(\tilde{X},\pi^{-1}\K_{Z}^{p}(\V)) \ar[r]^{} &
  H^{q}(\tilde{X},\pi^{-1} \Omega_{X}^{p}(\V)) \ar[r]^{} &
  H^{q}(\tilde{X}, \pi^{-1}\iota_{\ast}\Omega_{Z}^{p}(\iota^{\ast}\V))\ar[r]^{} &
  H^{q+1}(\tilde{X},\pi^{-1}\K_{Z}^{p}(\V)) \ar[r] &\cdots.}}
\end{equation}

Moreover, because the topological inverse image functor $\pi^{-1}$ is exact, applying $\pi^{-1}$ to \eqref{twist-exactseq-X} yields a short exact sequence of $\pi^{-1}\CO_{X}$-modules
$$
\xymatrix@C=0.5cm{
0 \ar[r]^{} & \pi^{-1}\K_{Z}^{p}(\V)
   \ar[r]^{} & \pi^{-1}\Omega_{X}^{p}(\V)
   \ar[r]^{} & \pi^{-1}\iota_{\ast}\Omega_{Z}^{p}(\iota^{\ast}\V)
   \ar[r]^{} & 0.}
$$
Via a straightforward checking, we can show that the blow-up diagram \eqref{blow-up-diagram}
induces a commutative diagram of short exact sequences
\begin{equation*}
\xymatrix@C=0.5cm{
 0 \ar[r]^{} &  \pi^{-1}\K_{Z}^{p}(\V) \ar[d]_{\pi^{\#}} \ar[r]^{} & \pi^{-1}\Omega_{X}^{p}(\V) \ar[d]_{\pi^{\#}} \ar[r]^{\iota^{\#}} &  \pi^{-1}\iota_{\ast}\Omega_{Z}^{p}(\iota^{\ast}\V)\ar[d]_{\rho^{\#}} \ar[r]^{} & 0 \\
0 \ar[r] &  \K_{E}^{p}(\tilde{\V}) \ar[r]^{} &
  \Omega_{\tilde{X}}^{p}(\tilde{\V}) \ar[r]^{\tilde{\iota}^{\#}} &
 \tilde{\iota}_{\ast}\Omega_{E}^{p}(\tilde{\iota}^{\ast}\tilde{\V}) \ar[r] &0, }
\end{equation*}
where the morphisms $(-)^{\#}$ are induced by the pullbacks of regular differential forms.
Taking the cohomology functor $H^{\bullet}(\tilde{X}, -)$ to the diagram above,
we obtain a commutative ladder of long exact sequences
\begin{equation}\label{long-sequ2}
\vcenter{
\xymatrix@C=0.3cm{
\cdots\ar[r]^{} & H^{q}(\tilde{X},\pi^{-1}\K_{Z}^{p}(\V))  \ar[d]_{} \ar[r]^{} &H^{q}(\tilde{X}, \pi^{-1}\Omega_{X}^{p}(\V)) \ar[d]_{} \ar[r]^{} & H^{q}(\tilde{X}, \pi^{-1}\iota_{\ast}\Omega_{Z}^{p}(\iota^{\ast}\V)) \ar[d]_{} \ar[r]^{} & H^{q+1}(\tilde{X},\pi^{-1}\K_{Z}^{p}(\V))\ar[d]_{} \ar[r]^{} & \cdots\\
 \cdots\ar[r] & H^{q}(\tilde{X},\K_{E}^{p}(\tilde{\V})) \ar[r]^{} &
  H^{q}(\tilde{X}, \Omega_{\tilde{X}}^{p}(\tilde{\V})) \ar[r]^{} &
  H^{q}(\tilde{X}, \tilde{\iota}_{\ast}\Omega_{E}^{p}(\tilde{\iota}^{\ast}\tilde{\V}))\ar[r]^{} &
  H^{q+1}(\tilde{X},\K_{E}^{p}(\tilde{\V})) \ar[r] &  \cdots}}
\end{equation}
Since $\iota$ and $\tilde{\iota}$ are closed inclusions, we have
\begin{equation}\label{Z}
H^{q}(X, \iota_{\ast}\Omega_{Z}^{p}(\iota^{\ast}\V))=
H^{q}(Z, \Omega_{Z}^{p}(\iota^{\ast}\V))
\end{equation}
and
\begin{equation}\label{E}
H^{q}(\tilde{X}, \tilde{\iota}_{\ast}\Omega_{E}^{p}
(\tilde{\iota}^{\ast}\tilde{\V}))=
H^{q}(E,\Omega_{E}^{p}
(\tilde{\iota}^{\ast}\tilde{\V})).
\end{equation}
From \eqref{long-sequ1}-\eqref{E},
we get the desired commutative ladder of long exact sequences
\begin{equation}\label{final-diagram}
\vcenter{
\xymatrix@C=0.5cm{
\cdots\ar[r]^{} & H^{q}(X, \K_{Z}^{p}(\V))  \ar[d]_{\pi^{\#}} \ar[r]^{} &H^{q}(X, \Omega_{X}^{p}(\V)) \ar[d]_{\pi^{\#}}^{} \ar[r]^{} & H^{q}(Z, \Omega_{Z}^{p}(\iota^{\ast}\V)) \ar[d]_{\rho^{\#}} \ar[r]^{} & H^{q+1}(X, \K_{Z}^{p}(\V))\ar[d]_{\pi^{\#}}\ar[r]^{} & \cdots\\
 \cdots\ar[r] & H^{q}(\tilde{X},\K_{E}^{p}(\tilde{\V})) \ar[r]^{} &
  H^{q}(\tilde{X}, \Omega_{\tilde{X}}^{p}(\tilde{\V})) \ar[r]^{\tilde{\iota}^{\#}} &
  H^{q}(E, \Omega_{E}^{p}(\tilde{\iota}^{\ast}\tilde{\V}))\ar[r]^{} &
  H^{q+1}(\tilde{X},\K_{E}^{p}(\tilde{\V})) \ar[r] &\cdots }}
\end{equation}

\paragraph{\textbf{Step 2}}
According to Theorem \ref{key-tech-prop},
we see that the first and the fourth column maps in \eqref{final-diagram} are isomorphisms.
Now we verify the injectivity of the second column map in \eqref{final-diagram}.
The basic idea used here is attributed to Deligne \cite[Proposition 4.3]{Del68}.

\begin{lem}\label{injective-lem}
For any integer $q\geq0$,
the induced morphism
\begin{equation*}
\pi^{\#}:
H^{q}(X, \Omega_{X}^{p}(\V))
\longrightarrow
H^{q}(\tilde{X}, \pi^{-1}\Omega_{X}^{p}(\V))
\longrightarrow
H^{q}(\tilde{X}, \Omega_{\tilde{X}}^{p}(\tilde{\V}))
\end{equation*}
is injective.
\end{lem}

\begin{proof}
Note that $\Omega_{X}^{n-p}(\V)$ is a locally free sheaf over $X$.
The pullback $\pi^{\ast}$ induces a natural morphism
\begin{equation}\label{Del1}
\pi^{\ast}\Omega_{X}^{n-p}(\V)
=
L\pi^{\ast}\Omega_{X}^{n-p}(\V)
\longrightarrow
\Omega_{\tilde{X}}^{n-p}(\tilde{\V}).
\end{equation}
Recall the definition of Grothendieck's duality functor $\pi^{!}$.
For any object $\E$ in the bounded derived category of coherent sheaves $\D^{b}(\Coh(X))$, the duality of $\E$ is defined to be
\begin{equation*}
\pi^{!}\E
:=
\omega_{\tilde{X}}\otimes \pi^{\ast}\omega_{X}^{\vee}\otimes L\pi^{\ast}\E
\cong
R\mathcal{H}om_{\tilde{X}}(L\pi^{\ast}\omega_{X}, L\pi^{\ast}\E\otimes \omega_{\tilde{X}}),
\end{equation*}
which is an object in $\D^{b}(\Coh(\tilde{X}))$.
Set $\E=\Omega_{X}^{p}(\V)$ and then we have
\begin{equation}\label{Del2}
\pi^{!}\Omega_{X}^{p}(\V)
=
R\mathcal{H}om_{\tilde{X}}(L\pi^{\ast}\Omega_{X}^{n-p}, \tilde{\V}\otimes \omega_{\tilde{X}}).
\end{equation}
According to the canonical isomorphisms
$$\label{Del3}
\Omega_{\tilde{X}}^{p}(\tilde{\V})
\cong (\Omega_{\tilde{X}}^{n-p})^{\vee}\otimes\tilde{\V}\otimes \omega_{\tilde{X}}
\cong R\mathcal{H}om_{\tilde{X}}(\Omega_{\tilde{X}}^{n-p}, \tilde{\V}\otimes \omega_{\tilde{X}})
$$
and also \eqref{Del2},
the morphism \eqref{Del1} gives rise to a morphism
\begin{equation}\label{Del4}
\Omega_{\tilde{X}}^{p}(\tilde{\V})
\longrightarrow
\pi^{!}\Omega_{X}^{p}(\V).
\end{equation}
In particular, since $\pi^{!}$ is the right adjoint functor of $R\pi_{\ast}$ there is a natural isomorphism
$$\label{Del5}
\mathrm{Hom}(R\pi_{\ast}\Omega_{\tilde{X}}^{p}(\tilde{\V}),
\Omega_{X}^{p}(\V))
\cong
\mathrm{Hom}(\Omega_{\tilde{X}}^{p}(\tilde{\V}), \pi^{!}\Omega_{X}^{p}(\V)).
$$
Consequently, there exists a morphism
$\mathrm{Tr}: R\pi_{\ast}\Omega_{\tilde{X}}^{p}(\tilde{\V}) \longrightarrow\Omega_{X}^{p}(\V)$ corresponding to the morphism \eqref{Del4};
furthermore, we have a composition morphism
\begin{equation}\label{Del6}
\mathrm{Tr}\circ \pi^{\#}:
\Omega_{X}^{p}(\V)
\longrightarrow
R\pi_{\ast}\pi^{-1} \Omega_{X}^{p}(\V)
\longrightarrow
R\pi_{\ast}\Omega_{\tilde{X}}^{p}(\tilde{\V})
\longrightarrow
\Omega_{X}^{p}(\V)
\end{equation}
in the derived category $\D^{+}(\CO_{X})$, and hence it is a morphism of locally free sheaves.
Note that $\tilde{X}-E$ is isomorphic to $X-Z$ under the blow-up morphism $\pi$.
It follows that the morphism \eqref{Del6} is the \emph{identity} on the dense open subset $X-Z$.
 As a result, the induced morphism of cohomology
\begin{equation*}\label{Del7}
\mathrm{Tr}\circ \pi^{\#}:
H^{q}(X, \Omega_{X}^{p}(\V))
\longrightarrow
H^{q}(\tilde{X}, \Omega_{\tilde{X}}^{p}(\tilde{\V}))
\longrightarrow
H^{q}(X, \Omega_{X}^{p}(\V))
\end{equation*}
is the identity and thus we are led to the conclusion that the morphism $\pi^{\#}$ is injective.
\end{proof}

Now we are in a position to state the abstract Hodge blow-up formula.
From Lemma \ref{injective-lem} and Theorem \ref{key-tech-prop},
the Four Lemma implies that $\rho^{\#}$ in \eqref{final-diagram} is injective too.
Note that \eqref{final-diagram} is a commutative  ladder of finite-dimensional $k$-vector spaces.
A standard diagram-chasing, such as \cite[Proposition 5.1]{RYY19b}, shows that $\tilde{\iota}^{\#}$ in \eqref{final-diagram} induces an isomorphism of $k$-vector spaces:
$$
\mathrm{coker}\,\pi^{\#}\cong\mathrm{coker}\,\rho^{\#},
$$
and thus by the commutative ladder \eqref{final-diagram}, we have the abstract blow-up formula:
\begin{equation}\label{abstrct-iso}
H^{q}(\tilde{X},\Omega_{\tilde{X}}^{p}(\tilde{\V}))
\cong H^{q}(X,\Omega_{X}^{p}(\V))\oplus \mathrm{coker}\,\rho^{\#}.
\end{equation}

\paragraph{\textbf{Step 3}}
We will give the projective bundle formula of Hodge cohomology of locally free sheaves which is indeed well-known to experts.
Consider the projective bundle $\rho:E\longrightarrow Z$.
Then there is a canonical isomorphism
\begin{equation}\label{essetial-pbformula}
\bigoplus_{i=0}^{c-1} \Omega_{Z}^{p-i}[-i]
\stackrel{\simeq}\longrightarrow
R\rho_{\ast}\Omega_{E}^{p}.
\end{equation}
In fact, one can show this by essentially using the higher direct images of the relative sheaf $\Omega_{E/Z}^{i}$ (\cite[XI, Theorem 1.1]{DK73}); for example, see \cite[page 22, (4.2.7)]{Gros85} for more details.
In general, let $\mathcal{W}$ be a locally free sheaf over $Z$.
Twist \eqref{essetial-pbformula} with $\mathcal{W}$ and then
the projection formula leads a canonical isomorphism
\begin{equation}\label{essetial-pbformula2}
\bigoplus_{i=0}^{c-1} \Omega_{Z}^{p-i}\otimes\mathcal{W}[-i]
\stackrel{\simeq}\longrightarrow
R\rho_{\ast}(\Omega_{E}^{p}\otimes\rho^{\ast}\mathcal{W}).
\end{equation}
Taking cohomology $H^{q}(Z, -)$ of \eqref{essetial-pbformula2},
we get the following isomorphisms
\begin{equation*}\label{proj-bundle}
\begin{aligned}
H^{q}(E, \Omega_{E}^{p}\otimes \rho^{\ast}\mathcal{W})
&\cong&
H^{q}(Z, R\rho_{\ast}(\Omega_{E}^{p}\otimes \rho^{\ast}\mathcal{W}))  \\
&\cong&
H^{q}(Z, \bigoplus_{i=0}^{c-1} \Omega_{Z}^{p-i}\otimes\mathcal{W}[-i])\\
&\cong&
\bigoplus_{i=0}^{c-1} H^{q-i}(Z, \Omega_{Z}^{p-i}\otimes\mathcal{W}).
\end{aligned}
\end{equation*}
Set $\mathcal{W}=\iota^{\ast}\V$ and then together with the abstract blow-up formula \eqref{abstrct-iso}, we are led to the final explicit blow-up formula
$$
H^{q}(\tilde{X},\Omega_{\tilde{X}}^{p}(\tilde{\V}))
\cong H^{q}(X,\Omega_{X}^{p}(\V))\oplus
\bigoplus_{i=1}^{c-1} H^{q-i}(Z, \Omega_{Z}^{p-i}\otimes\iota^{\ast}\V).
$$

\begin{rem}
It is of importance to notice that Lemma \ref{injective-lem} is slightly different from Deligne \cite[Proposition 4.3]{Del68}.
In \cite[Proposition 4.3]{Del68}, Deligne considered the induced morphism
$$
\pi^{\ast}:
H^{q}(X, \Omega_{X}^{p})
\rightarrow
H^{q}(\tilde{X}, \pi^{\ast}\Omega_{X}^{p})
\rightarrow
H^{q}(\tilde{X}, \Omega_{\tilde{X}}^{p}).
$$
The morphism above seems not compatible with the diagram \eqref{final-diagram} very well.
The main reason lies in the fact that the inverse image functor $\pi^{\ast}$ is not exact
and $\pi^{\ast}\Omega_{X}^{p}$ is not isomorphic to
$\pi^{-1}\Omega_{X}^{p}$ in general.
\end{rem}

%====================================================================

\section{Applications}\label{Appl.}

In this section, we mainly focus on the applications of Theorem \ref{mainthm} in positive characteristic to the blow-up invariance of the degeneracy of spectral sequences.

Let us fix several notations. Here $X$ is always a smooth proper variety over an algebraically closed field $k$ of positive characteristic.
Let $\iota:Z\hookrightarrow X$ be a smooth closed subvariety of codimension $c\geq 2$
and $\pi: \tilde{X}\longrightarrow X$ the blow-up of $X$ along $Z$.
%t
In the Subsections \ref{Hdg-bundle} and \ref{Alg.-deRham}, the index $l$ will always denote an arbitrary nonnegative integer in $[0, 2 \dim X]$.
\subsection{Hodge cohomology of locally free sheaves}\label{Hdg-bundle}
We will obtain the blow-up formula of total Hodge cohomology of locally free sheaves and discuss its applications.

\begin{defn}
Let $X$ be a smooth proper variety and $\V$ a locally free sheaf on $X$.
Denote by
$$
H_{\Hdg}^{l}(X; \V):=\bigoplus_{p+q=l} H^{q}(X, \Omega_{X}^{p}\otimes \V)
$$
the {\it $l$-th total Hodge cohomology of $X$ with coefficients in $\V$}.
In particular, if $\V=\CO_{X}$,
then we call $H_{\Hdg}^{l}(X):=H_{\Hdg}^{l}(X; \CO_{X})$ the {\it $l$-th total Hodge cohomology} of $X$.
Notice that the term \lq\lq total Hodge cohomology" here is often called \lq\lq Hodge cohomology" in many other literatures.
\end{defn}

Recently,
Achinger--Zdanowicz \cite{AZ17} obtained the {\it blow-up formula of total Hodge cohomology} by using Voevodsky's {\it blow-up formula of motives} (cf. \cite[(3.5.3)]{Voe00}).

\begin{prop}[{\cite[Corollary 2.8.(4)]{AZ17}}]\label{totalHodge-blowup}
There is an isomorphism
\begin{equation*}
H_{\Hdg}^{l}(\tilde{X})
\cong
H_{\Hdg}^{l}(X)
\oplus
\bigoplus_{i=1}^{c-1} H_{\Hdg}^{l-2i}(Z)
\end{equation*}
of total Hodge cohomology.
\end{prop}

As an application of Theorem \ref{mainthm},
we generalize this result to be as follows.

\begin{prop}\label{twisted-totalHodge-blowup}
There exists an isomorphism
\begin{equation*}
H_{\Hdg}^{l}(\tilde{X}; \pi^{\ast}\V)
\cong
H_{\Hdg}^{l}(X; \V)
\oplus
\bigoplus_{i=1}^{c-1} H_{\Hdg}^{l-2i}(Z; \iota^{\ast}\V)
\end{equation*}
of total Hodge cohomology of locally free sheaves.
\end{prop}

\begin{proof}
By the definition of total Hodge cohomology of locally free sheaves,
this is a direct consequence of Theorem \ref{mainthm}.
\end{proof}

Furthermore,
we will apply Proposition \ref{twisted-totalHodge-blowup} to the following interesting problem under blow-ups.
\begin{prob}[{cf. \cite[Question 2.1]{EO09}}] \label{EO-quest}
Let $X$ be a smooth projective variety over $k$.
Let $L$ be an invertible sheaf on $X$ and $m$ a positive integer such that $L^{\otimes m}\cong \CO_{X}$.
Is
\begin{equation}\label{ques-5.4}
\dim\, H_{\Hdg}^{l}(X; L^{\otimes j})=\dim\, H_{\Hdg}^{l}(X; L)
\end{equation}
for every $j$ relatively prime to $m$?
\end{prob}

\begin{rem}
In \cite[Proposition 3.5]{PR04},
in the case of $\mathrm{char}(k)=0$, Pink--Roessler showed that the answer to Problem \ref{EO-quest} is affirmative; see also a different proof of Esnault--Ogus \cite[Proposition 2.2]{EO09}.
In positive characteristic, Pink--Roessler posed \cite[Conjecture 5.1]{PR04}:
if $\dim\, X\leq \mathrm{char}(k)$ and $X$ is liftable over the ring $W_{2}(k)$ of $2$-Witt vectors, then Problem \ref{EO-quest} is true.
They proved it in this case for $(m, \mathrm{char}(k))=1$ (\cite[Theorem 3.2]{PR04}); see for example \cite{Cuo10} for a higher  dimensional generalization.
\end{rem}

In \cite[Remark 3.9]{EO09}, it was illustrated that, by Riemann--Roch theorem, Problem \ref{EO-quest} holds for smooth curves without any assumption.
In higher dimensional cases,
so far, we merely know that there is a positive answer of Esnault--Ogus \cite[Theorem 3.6]{EO09} for $m=\mathrm{char}(k)$ and $X$ ordinary.
In general, this problem seems difficult to be handled.

As an application of Proposition \ref{twisted-totalHodge-blowup},
one gets the following observation.

\begin{lem}\label{blowup-EO-quest}
With the same assumptions as in Problem \ref{EO-quest}, if the equality \eqref{ques-5.4} holds for $(X, L)$ and $(Z, \iota^{\ast}L)$,
then so does for $(\tilde{X}, \pi^{\ast}L)$.
\end{lem}

\begin{proof}
Suppose that there is a positive integer $m$ such that $(\pi^{\ast}L)^{\otimes m}\cong \CO_{\tilde{X}}$.
By the projection formula,
we have
$$
\CO_{X}\cong \pi_{\ast}\CO_{\tilde{X}}\cong \pi_{\ast} \pi^{\ast}L^{\otimes m}
\cong L^{\otimes m}\otimes \pi_{\ast}\CO_{\tilde{X}}\cong L^{\otimes m}
$$
and hence $(\iota^{\ast}L)^{\otimes m}\cong \CO_{Z}$.
By the hypothesis that \eqref{ques-5.4} holds for $(X, L)$ and $(Z, \iota^{\ast}L)$,
Proposition \ref{twisted-totalHodge-blowup} yields
\begin{eqnarray*}
\dim\, H_{\Hdg}^{l}(\tilde{X}; (\pi^{\ast}L)^{\otimes j})
&=&
\dim\, H_{\Hdg}^{l}(X; L^{\otimes j})+\sum_{i=1}^{c-1}\dim\, H_{\Hdg}^{l-2i}(Z; (\iota^{\ast}L)^{\otimes j}) \\
&=& \dim\, H_{\Hdg}^{l}(X; L)+\sum_{i=1}^{c-1}\dim\, H_{\Hdg}^{l-2i}(Z; \iota^{\ast}L) \\
&=& \dim\, H_{\Hdg}^{l}(\tilde{X}; \pi^{\ast}L)
\end{eqnarray*}
for relatively prime  $j$ and $m$.
Hence, the equality \eqref{ques-5.4} holds for $(\tilde{X}, \pi^{\ast}L)$.
\end{proof}

Using this lemma,
one can construct many new examples such that the equality \eqref{ques-5.4} holds.

\begin{example}\label{eg-eo}
With $X$ as in the example of \cite[Theorem 3.6]{EO09},
if $L^{\otimes \mathrm{char}(k)}\cong \CO_{X}$,
Lemma \ref{blowup-EO-quest} implies that
\eqref{ques-5.4} holds for the blow-up $(\tilde{X}, \pi^{\ast}L)$ of $(X, L)$ at points or smooth curves.
\end{example}

Specifically, in the three-dimensional case, one has
\begin{cor}
Let $X$ be a smooth projective threefold and $L$ an invertible sheaf on $X$.
Then the equality \eqref{ques-5.4} holds for $(\tilde{X}, \pi^{\ast}L)$ if and only if it holds for $(X, L)$.
\end{cor}

\begin{proof}
If the equality \eqref{ques-5.4} holds for $(X, L)$, then it holds for $(\tilde{X}, \pi^{\ast}L)$ by Lemma \ref{blowup-EO-quest} since it holds for points and smooth curves as shown in Example \ref{eg-eo}.

Conversely, suppose that \eqref{ques-5.4} holds for $(\tilde{X}, \pi^{\ast}L)$ and $(\pi^{\ast}L)^{\otimes m}\cong \CO_{\tilde{X}}$.
Recall again that \eqref{ques-5.4} holds for points and smooth curves, and thus Proposition \ref{twisted-totalHodge-blowup} gives
\begin{eqnarray*}
\dim\, H_{\Hdg}^{l}(X; L^{\otimes j})
&=&
\dim\, H_{\Hdg}^{l}(\tilde{X}; (\pi^{\ast}L)^{\otimes j})-\sum_{i=1}^{c-1}\dim\, H_{\Hdg}^{l-2i}(Z; (\iota^{\ast}L)^{\otimes j}) \\
&=& \dim\, H_{\Hdg}^{l}(\tilde{X}; \pi^{\ast}L)-\sum_{i=1}^{c-1}\dim\, H_{\Hdg}^{l-2i}(Z; \iota^{\ast}L) \\
&=& \dim\, H_{\Hdg}^{l}(X; L)
\end{eqnarray*}
for relatively prime $j$ and $m$.
\end{proof}

%----------------------------------------------------------------------------------------------------------------

\subsection{Algebraic de Rham cohomology}\label{Alg.-deRham}
Let $X$ be a smooth projective variety over $k$ and $\mathcal{E}$ an algebraic vector bundle on $X$ with an integrable connection.
We denote by $H_{\DR}^{l}(X/k; \mathcal{E})$ the {\it $l$-th algebraic de Rham cohomology} of $X$ with coefficients in $\E$ (see Grothendieck \cite{Gro66}).
We also have the {\it twisted Hodge--de Rham spectral sequence}
\begin{equation}\label{twisted-Hodge-to-de Rham}
E_{1}^{p,q}=H^{q}(X, \Omega_{X/k}^{p}\otimes \E)\Longrightarrow H_{\DR}^{p+q}(X/k; \E).
\end{equation}
Then the twisted Hodge--de Rham spectral sequence \eqref{twisted-Hodge-to-de Rham} degenerates at $E_{1}$ if and only if
$$
\dim\, H_{\Hdg}^{l}(X; \E)=\dim\, H_{\DR}^{l}(X/k; \E)
$$
for every integer $l\geq0$.
Naturally, we can ask the following:
\begin{prob}\label{alg-twist-deRham-quest}
Is there an isomorphism
\begin{equation}\label{ques-5.9}
H_{\DR}^{l}(\tilde{X}/ k; \pi^{\ast}\E)
\cong
H_{\DR}^{l}(X/k; \E)
\oplus
\bigoplus_{i=1}^{c-1} H_{\DR}^{l-2i}(Z/k; \iota^{\ast}\E)
\end{equation}
of algebraic de Rham cohomology in positive characteristic?
\end{prob}

\begin{rem}
Over the complex number field,
the answer to this problem is a consequence of Serre's GAGA, Grothendieck--Deligne's comparison theorem and \cite[Theorem 1.1]{CY19}.
\end{rem}

If the answer to Problem \ref{alg-twist-deRham-quest} is affirmative,
then Proposition \ref{twisted-totalHodge-blowup} yields the following.

\begin{cor}\label{cor1}
Suppose that \eqref{ques-5.9} holds.
Then the $E_{1}$-degeneracy of the  twisted Hodge--de Rham spectral sequence  \eqref{twisted-Hodge-to-de Rham}
holds for $(X, \E)$ and $(Z, \iota^{\ast}\E)$
if and only if
so does for $(\widetilde{X}, \pi^{\ast}\E)$.
\end{cor}

\begin{proof}
Suppose that the isomorphism \eqref{ques-5.9} holds.
Then, by Proposition \ref{twisted-totalHodge-blowup},
we have
{
\begin{eqnarray*}
&&\underbrace{\dim\,H_{\DR}^{l}(\tilde{X}/ k; \pi^{\ast}\E)-\dim\,H_{\Hdg}^{l}(\tilde{X}; \pi^{\ast}\E)}_{\leq 0}\\
&=&\underbrace{\dim\,H_{\DR}^{l}(X/ k; \E)-\dim\,H_{\Hdg}^{l}(X; \E)}_{\leq 0}
+
\sum_{i=1}^{c-1}\big( \underbrace{\dim\,H_{\DR}^{l}(Z/ k; \iota^{\ast}\E)-
\dim\,H_{\Hdg}^{l-2i}(Z; \iota^{\ast}\E)}_{\leq 0}\big).
\end{eqnarray*}
}
Hence, the corollary follows.
\end{proof}

Finally, we say a few words on the birational invariance for $E_{1}$-degeneracy of the Hodge--de Rham spectral sequence \eqref{Hodge-to-de Rham}.
In positive characteristic,
Mumford \cite{Mum61} gave several explicit examples of smooth projective surfaces
with non-closed global $1$-forms;
this means the exterior derivative
$$d:H^{0}(X,\Omega_{X}^{1})\longrightarrow H^{0}(X,\Omega_{X}^{2})$$ is non-zero,
which implies that the Hodge--de Rham spectral sequence \eqref{Hodge-to-de Rham} does not degenerate at $E_{1}$.
This also means that, in general, the $E_{1}$-degeneracy of Hodge--de Rham spectral sequence  \eqref{Hodge-to-de Rham} is not a birational property of smooth projective varieties of dimension $\geq 4$.

Furthermore, we have the following observation for smooth projective surfaces.

\begin{cor}
The $E_{1}$-degeneracy of the Hodge--de Rham spectral sequence  \eqref{Hodge-to-de Rham} is a birational property of smooth projective surfaces over $k$ of positive characteristic.
\end{cor}

\begin{proof}
Note that the weak factorization theorem holds for smooth projective surfaces, i.e.,
any birational map between smooth projective surfaces is factorized by finite sequences of blow-ups and blow-downs along points (cf. \cite[Chapter V, Theorem 5.5]{Har77}).
Therefore, the corollary follows from \cite[Corollary 2.9.(1)]{AZ17}.
\end{proof}

In positive characteristic,
it is known that the Hodge--de Rham spectral sequence degenerates at $E_{1}$ for smooth projective curves (see \cite{DL87}).
Hence, there is a natural problem for threefolds.

\begin{prob}
Is the $E_{1}$-degeneracy of the Hodge--de Rham spectral sequence \eqref{Hodge-to-de Rham} a birational property for smooth proper threefolds over $k$ of positive characteristic?
\end{prob}

%----------------------------------------------------------------------------------------------------------------

\subsection{Hochschild--Kostant--Rosenberg theorem in positive characteristic}\label{Appl.-HKR}
In this subsection, we will consider the blow-up invariance of the Hochschild--Kostant--Rosenberg theorem in positive characteristic.

Throughout this subsection, the index $l$ denotes any integer in $[-\dim\, X,\ \dim\, X]$.
\begin{defn}
Let $X$ be a smooth proper variety over $k$.
The {\it Hochschild complex of $X$} is defined as
$$
\HC_{\bullet}(X):=L\Delta^{\ast}(\Delta_{\ast}\CO_{X})
$$
where $\Delta:X\longrightarrow X\times_{k} X$ is the diagonal.
For a locally free sheaf $\V$ on $X$, let
$$
\mathrm{HH}_{l}(X; \V):=H^{-l}(X, \HC_{\bullet}(X)\otimes \V)
$$
be the {\it Hochschild homology of $X$ with values in $\V$}.
In particular,
if $\V=\CO_{X}$, then
$$
\mathrm{HH}_{l}(X):=\mathrm{HH}_{l}(X; \CO_{X})
$$
is called the {\it Hochschild homology of $X$}.
\end{defn}

We have the {\it Hochschild--Kostant--Rosenberg spectral sequence} (for short {\it HKR spectral sequence})
$$
E_{2}^{p,q}=H^{q}(X, \Omega_{X}^{p})\Longrightarrow \mathrm{HH}_{p-q}(X),
$$
with the differential $d_{r}$ having bi-degree $(r-1, r)$.
Hence, with $h^{q}(X, \Omega_{X}^{p}):=\dim_k H^{q}(X, \Omega_{X}^{p})$, one has the following inequality
$$
\dim\, \mathrm{HH}_{l}(X)\leq \sum_{p-q=l}h^{q}(X, \Omega_{X}^{p})
$$
for any integer $l\in[-\dim\, X,\dim\, X]$; furthermore, the equality holds if and only if the HKR spectral sequence degenerates at $E_{2}$.
Following \cite{AV17}, we also say that $X$ satisfies the {\it weak HKR theorem} if the $E_{2}$-degeneracy of the HKR spectral sequence holds on $X$.
We say that $X$ satisfies the {\it strong HKR theorem}
if there exists an isomorphism
$$
\HC_{\bullet}(X)\cong \bigoplus_{p=0}^{\dim\, X} \Omega_{X}^{p}[p],
$$
in the derived category $\D^{b}(\Coh(X))$.
Hence, the strong HKR theorem for $X$ implies the weak HKR theorem on $X$.
Moreover, the strong HKR theorem implies the {\it Hochschild--Kostant--Rosenberg theorem}, namely, there is an isomorphism
$$
\mathrm{HH}_{l}(X)\cong\bigoplus_{p=0}^{\dim\, X} H^{p-l}(X, \Omega_{X}^{p})
$$
for every integer $-\dim\, X\leq l\leq \dim\, X$.
It is now well-known that the strong HKR theorem holds in characteristic zero (cf. \cite{Swa96,Yek02, TV11}).
In positive characteristic, there is a natural problem (cf. \cite[Question 1.1]{AV17}):

\begin{prob}\label{quest-HKR}
Is the strong HKR theorem true in positive characteristic?
\end{prob}

Partially, the answer to this problem is affirmed by \cite[Theorem 4.8]{Yek02} and \cite[Corollary 1.5]{AV17} as follows.

\begin{lem}\label{strong-HKR}
Let $X$ be a smooth proper variety over $k$.
If $\mathrm{char}(k)\geq \dim\, X$,
then $X$ satisfies the strong HKR theorem.
\end{lem}

In summary, combining with Theorem \ref{mainthm},
we have the following result.

\begin{cor}
Assume that $\mathrm{char}(k)\geq \dim\, X$. With the same setting as in Theorem \ref{mainthm},
there is an isomorphism
\begin{equation*}
\mathrm{HH}_{l}(\tilde{X}; \pi^{\ast}\V)
\cong
\mathrm{HH}_{l}(X; \V)
\oplus
\mathrm{HH}_{l}(Z; \iota^{\ast}\V)^{\oplus (c-1)}
\end{equation*}
for $-\dim\, X\leq l\leq \dim\, X$.
\end{cor}

\begin{proof}
By Lemma \ref{strong-HKR},
for any smooth proper variety $Y$,
we have an isomorphism
$$
\HC_{\bullet}(Y)\otimes \mathcal{V}
\cong
\bigoplus_{p=0}^{\dim\, Y} \Omega_{Y}^{p} \otimes\mathcal{V} [p]
$$
in the derived category.
Taking cohomology yields
$$
\mathrm{HH}_{l}(Y; \mathcal{V})
\cong
\bigoplus_{p-q=l} H^{q}(Y, \Omega_{Y}^{p}\otimes \mathcal{V}).
$$
Combining this with Theorem \ref{mainthm},
we obtain
\begin{eqnarray*}
\mathrm{HH}_{l}(\tilde{X}; \pi^{\ast}\V)
&\cong&
\bigoplus_{p-q=l} H^{q}(\tilde{X}, \Omega_{\tilde{X}}^{p}\otimes\pi^{\ast}\V) \\
&\cong &
\bigoplus_{p-q=l}  \Big( H^{q}(X, \Omega_{X}^{p}\otimes \V)\oplus \bigoplus_{i=1}^{c-1} H^{q-i}(Z, \Omega_{Z}^{p-i}\otimes\iota^{\ast}\V)  \Big)  \\
&\cong &
\mathrm{HH}_{l}(X; \V) \oplus \mathrm{HH}_{l}(Z; \iota^{\ast}\V)^{\oplus c-1}
\end{eqnarray*}
for any $-\dim\, X\leq l\leq \dim\, X$.
\end{proof}

In the rest of this subsection,
we shall study the blow-up invariance of the $E_{2}$-degeneracy of the HKR spectral sequence.
To this end, we start with the following result.

\begin{lem}\label{Hochschild-blowup}
For every integer $-\dim\, X\leq l\leq \dim\, X$,
there holds an isomorphism of Hochschild homology
$$
\mathrm{HH}_{l}(\tilde{X})
\cong
\mathrm{HH}_{l}(X)
\oplus
\mathrm{HH}_{l}(Z)^{\oplus (c-1)}.
$$
\end{lem}

\begin{proof}
It is known that,
for any smooth proper variety $Y$,
$\mathrm{HH}_{l}(Y)\cong\mathrm{HH}_{l}(\D^{b}(Y))$ (cf. \cite[Theorem 4.5]{Kuz09}).
By Orlov's blow-up formula \cite{Orl93} (cf. \cite[Proposition 11.18]{Huy06}),
we have a semiorthogonal decomposition
$$
\D^{b}(\tilde{X})
=
\langle L\pi^{\ast}\D^{b}(X), \tilde{\iota}_{\ast}(\CO_{E}(1)\otimes\rho^{\ast}\D^{b}(Z)),\ldots,\tilde{\iota}_{\ast}(\CO_{E}(c-1)\otimes \rho^{\ast}\D^{b}(Z)) \rangle.
$$
Since the functors $\tilde{\iota}_{\ast}(\CO_{E}(s)\otimes \rho^{\ast}-)$ for $1\leq s\leq c-1$ and $L\pi^{\ast}$ are full-faithful, $\tilde{\iota}_{\ast}(\CO_{E}(s)\otimes \rho^{\ast}\D^{b}(Z))$ is equivalent to $\D^{b}(Z)$ and $L\pi^{\ast}\D^{b}(X)$ is equivalent to $\D^{b}(X)$ as triangulated categories.
Therefore, the corollary follows from Kuznetsov's additivity for Hochschild homology (\cite[Theorem 7.3]{Kuz09}).
\end{proof}

\begin{cor}\label{HKR-blowup}
The weak HKR theorem holds for $X$ and $Z$ if and only if so does for $\tilde{X}$.
\end{cor}

\begin{proof}
By Lemma \ref{Hochschild-blowup} and \eqref{Hodge-bwf},
we have
\begin{eqnarray*}
&&\underbrace{\dim\,\mathrm{HH}_{l}(\tilde{X})-\sum_{p-q=l} h^{q}(\tilde{X},\Omega_{\tilde{X}}^{p})}_{\leq 0}\\
&=&\underbrace{\dim\,\mathrm{HH}_{l}(X)-\sum_{p-q=l} h^{q}(X,\Omega_{X}^{p})}_{\leq 0} +
\sum_{i=1}^{c-1}\big( \underbrace{\dim\,\mathrm{HH}_{l}(Z)-
\sum_{p-q=l} h^{q-i}(Z,\Omega_{Z}^{p-i}}_{\leq 0}\big).
\end{eqnarray*}
Consequently, this corollary follows from the definition of the weak HKR theorem.
\end{proof}

In particular, we have the following.

\begin{cor}
If $\dim\, Z \leq \mathrm{char}(k)<\dim\, X$,
then the weak HKR theorem holds for $X$ if and only if it holds for $\tilde{X}$.
\end{cor}

\begin{example}
The weak HKR theorem holds for smooth complete intersections in $\mathbb{P}^{N}$;
see for example \cite[Example 1.7]{AV17}.
Therefore, for instance, in $\mathrm{char}(k)=2$,
based on the above corollary and \cite[Example 1.4]{AV17},
one may construct many new examples of smooth projective varieties satisfying the weak HKR theorem by blowing up along points, curves or smooth surfaces.
\end{example}

Naturally, one may ask the following problem.

\begin{prob}
Suppose that $X$ is a smooth proper variety over $k$ and $\dim\, X-2 \leq \mathrm{char}(k)<\dim\, X$.
Is the weak HKR theorem a birational property of $X$?
\end{prob}

\begin{rem}\label{construct-example}
In \cite{ABM19},
Antieau--Bhatt--Mathew will show that the HKR spectral sequence does not generally degenerate at $E_2$ in the case of $\dim \,X=2\cdot\mathrm{char}(k)>0$.
This also gives a negative answer to Problem \ref{quest-HKR}.
Based on their examples, Corollary \ref{HKR-blowup} can provide more examples such that the HKR spectral sequence does not generally degenerate at $E_{2}$.
\end{rem}

Finally, we have the following observation to construct more higher odd-dimensional examples
such that the HKR spectral sequence does not generally degenerate at $E_{2}$.

\begin{cor}
Let $X$ be a smooth proper variety over $k$ and  $\E$ a locally free sheaf of rank $c$ on $X$.
Then the weak HKR theorem holds for the projective bundle $\mathbb{P}(\E)$ if and only if it  holds on $X$.
\end{cor}

\begin{proof}
Similar to Lemma \ref{Hochschild-blowup},
by Orlov's projective bundle formula (\cite{Orl93} or \cite{Huy06}) and Kuznetsov's additivity of Hochschild homology, we have
$$
\mathrm{HH}_{l}(\mathbb{P}(\E))
\cong
\mathrm{HH}_{l}(X)^{\oplus c}.
$$
As a result, we get
\begin{equation*}
\dim\,\mathrm{HH}_{l}(\mathbb{P}(\E))-\sum_{p-q=l} h^{q}(\mathbb{P}(\E),\Omega_{\mathbb{P}(\E)}^{p})
=c\cdot \big(\dim\,\mathrm{HH}_{l}(X)-\sum_{p-q=l} h^{q}(X,\Omega_{X}^{p})\big).
\end{equation*}
Thus, this corollary follows from the definition of the weak HKR theorem.
\end{proof}

\begin{rem}\label{rem5.25}
Combining this with Antieau--Bhatt--Mathew's examples,
one can obtain odd-dimensional ($\geq 5$) smooth proper varieties such that the HKR spectral sequence does not degenerate at $E_{2}$.
\end{rem}

%==================================================================

\section*{Acknowledgements}
This work started when the first three authors were visiting Institut Fourier (Math\'{e}matiques) at Universit\'{e} de Grenoble-Alpes, Departments of Mathematics at Universit\`{a} degli Studi di Milano and Cornell University, respectively; they would like to thank those institutes for the hospitality and providing the wonderful working environment.
Last but not least, all the authors sincerely thank Professor V. Navarro Aznar for pointing out the paper \cite{Hua01} to them and answering their question on Lemma \ref{iso1-2-3} (iii), and Professor K. R\"ulling for his important comment on their proof of the main theorem, and also Professor B. Antieau for sending us their paper \cite{ABM19}.
Finally, the authors are grateful to the anonymous referee for useful comments and suggestions.
S. Rao is partially supported by the National Nature Science Foundation of China (Grant No. 11671305, 11771339, 11922115) and the Fundamental Research Funds for the Central Universities (Grant No. 2042020kf1065).
S. Yang and X.-D. Yang are partially supported by the National Nature Science Foundation of China (Grant No. 11701414, 11571242,11701051) and the China Scholarship Council.
X. Yu is partially supported by the National Nature Science Foundation of China (Grant No. 11701413, 11831013).

%==================================================================


\begin{thebibliography}{GP}

\bibitem{AKMW02}
D. Abramovich, K. Karu, K. Matsuki, and J. W{\l}odarczyk,
{\it Torification and factorization of birational maps},
J. Amer. Math. Soc. {\bf 15}  (2002), 531--572.

\bibitem{AZ17}
P. Achinger and M. Zdanowicz,
{\it Some elementary examples of non-liftable varieties},
Proc. Amer. Math. Soc. {\bf 145} (2017), 4717--4729.

\bibitem{ABM19}
B. Antieau, B. Bhatt, and A. Mathew,
{\it Counterexamples to Hochschild--Kostant--Rosenberg in characteristic $p$},
Forum of Mathematics, Sigma {\bf 49} (2021), 1--26.

\bibitem{AB19}
B. Antieau and D. Bragg,
{\it Derived invariants from topological Hochschild homology},
\href{https://arxiv.org/abs/1906.12267v1}{arXiv:1906.12267}, to appear in Algebraic Geometry.

\bibitem{AV17}
B. Antieau and G. Vezzosi,
{\it A remark on the Hochschild--Kostant--Rosenberg theorem in characteristic $p$},
Ann. Sc. Norm. Super. Pisa Cl. Sci. {\bf XX} (3)(2020), 1135--1145.

\bibitem{Bar97}
L. Barbieri-Viale,
{\it $\mathscr{H}$-cohomologies versus algebraic cycles},
Math. Nachr. {\bf 184} (1997), 5--57.

\bibitem{BO74}
S. Bloch and  A. Ogus,
{\it Gersten's conjecture and the homology of schemes},
Ann. Sci. Ecole Norm. Sup. {\bf 7}, 181--202 (1974).

\bibitem{Bot57}
R. Bott,
{\it Homogeneous vector bundles},
Ann. of Math. {\bf 66} (1957), 203--248.

\bibitem{CR11}
A. Chatzistamatiou and  K. R\"{u}lling,
{\it Higher direct images of the structure sheaf in positive characteristic},
Algebra Number Theory {\bf 5} (2011), 693--775.


\bibitem{CY19}
Y. Chen and S. Yang,
{\it On the blow-up formula of twisted de Rham cohomology},
Ann. Glob. Anal. Geom. {\bf 56} (2019), 277--290.

\bibitem{Cuo10}
D. T. Cuong,
{\it Hodge cohomology of \'{e}tale Nori finite vector bundles},
Int. Math. Res. Not. {\bf 2010}  (2010),  320--333.

\bibitem{Del68}
P. Deligne,
{\it Th\'{e}or\`{e}me de Lefschetz et crit\`{e}res de d\'{e}gen\'{e}r\'{e}scence de suites spectrals},
Publ. Math Inst. Hautes \'{E}tudes Sci.  {\bf 35} (1968), 259--277.

\bibitem{DGMS75}
P. Deligne, P. Griffiths, J. Morgan, and D. Sullivan,
{\it Real homotopy theory of K\"{a}hler manifolds},
Invent. Math. {\bf 29} (1975), 245--274.

\bibitem{DK73}
P. Deligne and N. Katz,
{\it Groupes de monodromie en g\'{e}om\'{e}trie alg\'{e}brique (SGA 7 II)},
Lecture Notes in Mathematics, {\bf 340}, Springer, Berlin, Heidelberg, New York (1973).

\bibitem{DL87}
P. Deligne and L. Illusie,
{\it Rel\`{e}vements modulo $p^{2}$ et d\'{e}composition du complexe de de Rham},
Invent. Math. {\bf 89} (1987), 247--270.

\bibitem{Dem12}
J.-P. Demailly,
{\it Complex analytic and differential geometry},
\href{https://www-fourier.ujf-grenoble.fr/~demailly/manuscripts/agbook.pdf}{https://www-fourier.ujf-grenoble.fr/~demailly/manuscripts/agbook.pdf}.

\bibitem{EO09}
H. Esnault and A. Ogus,
{\it Hodge cohomology of invertible sheaves},
In: de Jeu R, Lewis J  D.  eds.
Motives and algebraic cycles.
Fields Inst. Commun., {\bf 56}, Amer. Math. Soc., Providence, RI, 83--91 (2009).

\bibitem{Gros85}
M. Gros,
{\it Classes de Chern et classes de cycles en cohomologie de Hodge--Witt logarithmique}
Bull. Soc. Math. France M\'{e}moire {\bf 21} (1985), 1--87.

\bibitem{Gro66}
A. Grothendieck,
{\it On the de Rham cohomology of algebraic varieties},
Publ. Math. Inst. Hautes \'{E}tudes Sci. {\bf 29} (1966), 95--103.




\bibitem{GNA02}
F. Guill\'{e}n and V. Navarro Aznar,
{\it Un crit\`{e}re d\'{e}xtension des foncteurs d\'{e}finis sur les sch\'{e}mas lisses},
Publ. Math. Inst. Hautes \'{E}tudes Sci. {\bf 95} (2002), 1--91.

\bibitem{Har77}
R. Hartshorne,
{\it Algebraic geometry},
Graduate Texts in Mathematics {\bf 52},
Springer-Verlag, New York, (1977).

\bibitem{Hir64}
H. Hironaka,
{\it Resolution of singularities of an algebraic variety over a field of characteristic zero I, II},
 Ann. of Math. {\bf 79} (1964), 109--203; 205--326.

\bibitem{Hua01}
I.C. Huang,
{\it Cohomology of projective space seen by residual complex},
Trans. Amer. Math. Soc. {\bf 353} (2001), 3097--3114.

\bibitem{Huy06}
D. Huybrechts,
{\it Fourier--Mukai transforms in algebraic geometry},
Oxford Mathematical Monographs,
Oxford University Press, Oxford (2006).

\bibitem{Ive86}
B. Iversen,
{\it Cohomology of sheaves},
Universitext. Springer-Verlag, Berlin (1986).

\bibitem{KS94}
M. Kashiwara and P. Schapira,
{\it Sheaves on manifolds},
Grundlehren Math. Wiss. {\bf 292}, Springer, Berlin (1994).

\bibitem{Kuz09}
A. Kuznetsov,
{\it Hochschild homology and semiorthogonal decompositions},
\href{https://arxiv.org/abs/0904.4330v1}{arXiv:0904.4330}.

\bibitem{Laz04}
R. Lazarsfeld,
{\it Positivity in algebraic geometry I, Classical setting: line bundles and linear series},
Ergebnisse der Mathematik und ihrer Grenzgebiete. 3. Folge.
A Series of Modern Surveys in Mathematics, {\bf 48}. Springer-Verlag, Berlin, (2004).

\bibitem{LS14}
C. Liedtke and M. Satriano,
{\it On the birational nature of lifting},
Adv. Math. {\bf 254} (2014), 118--137.


\bibitem{Mum61}
D. Mumford,
{\it Pathologies of modular algebraic surfaces},
Amer. J. Math. {\bf 83} (1961), 339--342.

\bibitem{Orl93}
D. Orlov,
{\it Projective bundles, monoidal transformations, and derived categories of coherent sheaves},
Russ. Acad. Sci. Izv. Math. {\bf 41} (1993), 133--141.

\bibitem{PR04}
R. Pink and D. Roessler,
{\it A conjeture of Beauville and Catanese revisited},
Math. Ann. {\bf 330} (2004), 293--308.

\bibitem{RYY19a}
S. Rao, S. Yang, and X.D. Yang,
{\it Dolbeault cohomologies of blowing up complex manifolds},
J. Math. Pures Appl. {\bf 130} (2019), 68--92.

\bibitem{RYY19b}
S. Rao, S. Yang, and X.D. Yang,
{\it Dolbeault cohomologies of blowing up complex manifolds II: bundle-valued cases},
J. Math. Pures Appl. {\bf 133} (2020), 1--38.

\bibitem{Ser56}
J.-P. Serre,
{\it G\'{e}om\'{e}trie alg\'{e}brique et g\'{e}om\'{e}trie analytique},
Ann. Inst. Fourier (Grenoble) {\bf 6} (1956), 1--42.


\bibitem{Ste18}
J. Stelzig,
{\it The double complex of a blow-up},
Int. Math. Res. Not. {\bf 2021} (2021), 10731--10744.


\bibitem{Swa96}
R.-G. Swan,
{\it Hochschild cohomology of quasiprojective schemes},
J. Pure Appl. Alg. {\bf 110}  (1996), 57--80.


\bibitem{TV11}
B. To\"{e}n and G. Vezzosi,
{\it Alg\'{e}bres simpliciales $S^1$-\'{e}quivariantes, th\'{e}orie de de Rham et th\'{e}or\'{e}mes HKR multiplicatifs},
Compos. Math. {\bf 147} (2011), 1979--2000.


\bibitem{Voe00}
V. Voevodsky,
{\it Triangulated categories of motives over a field},
Cycles, transfers, and motivic homology theories,
Ann. of Math. Stud., vol. {\bf 143}, Princeton Univ. Press,
Princeton, NJ, 188--238 (2000).

\bibitem{Voi02}
C. Voisin,
{\it Hodge theory and complex algebraic geometry I},
Cambridge Studies in Advanced Mathematics, {\bf 76}. Cambridge University
Press, Cambridge, (2002).

\bibitem{Wlo03}
J. W{\l}odarczyk,
{\it Toroidal varieties and the weak factorization theorem},
Invent. Math. {\bf 154} (2003), 223--331.

\bibitem{YY17}
S. Yang and X.D. Yang,
{\it Bott--Chern blow-up formulae and the bimeromorphic invariance of the $\partial\bar{\partial}$-Lemma for threefolds},
Trans. Amer. Math. Soc. {\bf 373} (2020), 8885--8909.

\bibitem{Yek02}
A. Yekutieli,
{\it The continuous Hochschild cochain complex of a scheme},
Canad. J. Math. {\bf 54} (2000), 1319--1337.

\end{thebibliography}
\end{document}